\documentclass[11pt]{amsart}

\usepackage{graphicx, amsmath, amssymb, amscd, amsthm, euscript, psfrag, amsfonts,bm,color}

\usepackage[dvipsnames]{xcolor}
\DeclareMathAlphabet{\mathpzc}{OT1}{pzc}{m}{it}
\usepackage{bbm}

\def\mainfile
\usepackage{amscd}
\usepackage[latin1]{inputenc}

\newtheorem{thm}[equation]{Theorem}

\newtheorem{Not}[equation]{Notation}
\newtheorem{rmk}[equation]{Remark}
\newtheorem{Rmk}[equation]{Remark}

\newtheorem{prop}[equation]{Proposition}

\newtheorem{cor}[equation]{Corollary}
\newtheorem{lem}[equation]{Lemma}\newtheorem{Lem}[equation]{Lemma}
\newtheorem{lemma}[equation]{Lemma}
\newtheorem{Def}[equation]{Definition}

\numberwithin{equation}{section}

\numberwithin{equation}{section}

\newcommand{\be}{begin{equation}}

\newcommand{\bH}{\mathbb H}

\newcommand{\e}{{\varepsilon}}

\newcommand{\N}{\mathbb{N}}
\renewcommand{\c}{\mathbb{C}}
\newcommand{\br}{\mathbb{R}}
\newcommand{\bbr}{\mathbb{R}}

\newcommand{\ba}{\backslash}

\newcommand{\Cal}{\mathcal}

\newcommand{\SL}{\operatorname{SL}}\newcommand{\PSL}{\operatorname{PSL}}

\newcommand{\bp}{\begin{pmatrix}}
\newcommand{\ep}{\end{pmatrix}}

\renewcommand{\bp}{{\rm bp}}

\newcommand{\SO}{\operatorname{SO}}

\newcommand{\T}{\operatorname{T}}

\newcommand{\vol}{\operatorname{vol}}

\newcommand{\Vol}{\op{Vol}}

\renewcommand{\epsilon}{\varepsilon}

\newcommand{\op}{\operatorname}

\renewcommand{\setminus}{-}
\newcommand{\Lie}{{\rm Lie}}
\renewcommand{\be}{\begin{equation}}
\newcommand{\ee}{\end{equation}}
\newcommand{\ben}{\begin{enumerate}}
\newcommand{\een}{\end{enumerate}}
\newcommand{\Ad}{\operatorname{Ad}}

\newcommand{\F}{\operatorname{F}\!}
\newcommand{\RF}{\operatorname{RF}\!}

\newcommand{\xt}{\tilde{x}}

\newcommand{\RFM}{\RF {M}}

\newcommand{\core}{\text{core}}

\newcommand{\hull}{\op{hull}}

\newcommand{\Sone}{M_0}

\newcommand{\sM}{{M}}
\newcommand{\FM}{\F {\sM}}

\newcounter{consta}
\renewcommand{\theconsta}{{D_{\arabic{consta}}}}
\newcounter{constb}[section]

\newcounter{constc}

\newcommand{\consta}{\refstepcounter{consta}\theconsta}

\newcommand{\La}{\Lambda}
\newcommand{\cal}{\mathcal}
\newcommand{\ds}{\delta_Y}
\newcommand{\inj}{\op{inj}}
\renewcommand{\core}{\op{core}}
\newcommand{\area}{\op{area}}
\newcommand{\tarea}{\op{area}_t}
\newcommand{\rdt}{\eta_0}

\newcommand{\Mt}{M_{\rm cpt}}
\begin{document}

\title[]{Isolations of geodesic planes in the frame bundle of a hyperbolic $3$-manifold}

\author{Amir Mohammadi}
\address{Mathematics Department, UC San Diego, 9500 Gilman Dr, La Jolla, CA 92093}

\email{ammohammadi@ucsd.edu}

\author{Hee Oh}
\address{Mathematics department, Yale university, New Haven, CT 06520 
and Korea Institute for Advanced Study, Seoul, Korea}

\email{hee.oh@yale.edu}
\thanks{The authors were supported in part by NSF Grants}


\begin{abstract} 
We present a quantitative isolation property of the lifts of properly immersed geodesic planes in the frame bundle of a geometrically finite hyperbolic $3$-manifold. 
Our estimates are polynomials in the tight areas and Bowen-Margulis-Sullivan densities of geodesic planes, with degree given by the modified critical exponents.
\end{abstract}
\maketitle
\tableofcontents
\section{Introduction}\label{sec:intro}
Let $\bH^3$ denote the hyperbolic $3$-space, and let $G:=\PSL_2(\c)$, which can be identified with the group $ \op{Isom}^+(\bH^3)$ of all orientation preserving isometries of $\bH^3$.
 Any complete orientable hyperbolic $3$-manifold can be  presented as a quotient 
 $M=\Gamma\ba \bH^3$ where $\Gamma$ is a torsion-free discrete subgroup of  $G$. 
 An oriented geodesic plane in $M$ is the image of a totally geodesic immersion of the hyperbolic plane $\bH^2\subset \bH^3$ equipped with an orientation under the quotient map $\bH^3\to \Gamma\ba \bH^3$. In this paper, all geodesic planes are assumed to be oriented.
Set  $X:=\Gamma\ba G$.
Via the identification of $X$ with the oriented frame bundle $\FM$, a geodesic plane in $M$ arises as the image of
a unique $\PSL_2(\br)$-orbit under the base point projection map $$\pi: X\simeq \FM \to \sM.$$
Moreover a {\it properly immersed} geodesic plane in $M$ corresponds to a {\it closed} $\PSL_2(\br)$-orbit in $X$.

Setting $H:=\PSL_2(\br)$, the main goal of this paper is to obtain a quantitative isolation result for closed $H$-orbits in  $X$ when $\Gamma$ is a 
 {\em geometrically finite} group.
{Fix a left invariant Riemannian metric on $G$, which projects to the hyperbolic metric on $\bH^3$.
This induces the distance $d$ on $X$ so that the  canonical projection $G\to X$  is a local isometry.  We use this Riemannian structure on $G$  to define the volume of a closed $H$-orbit in $X$.
For a closed subset $S\subset X$ and $\e>0$, $B(S, \e)$ denotes the $\e$-neighborhood of  $S$.

\medskip

\noindent{\bf The case when $M$ is compact.} 
We first state the result for compact hyperbolic $3$-manifolds. In this case,   Ratner \cite{Ra} and Shah \cite{Sh}  independently showed that every $H$-orbit  is either compact or dense in $X$. 
Moreover, there  are only countably many compact $H$-orbits in $X$. 
Mozes and Shah~\cite{MS} proved that an infinite sequence of compact $H$-orbits becomes equidistributed  in $X$.
Our questions concern the following quantitative isolation property: for given compact $H$-orbits $Y$ and $Z$ in $X$,

\begin{enumerate}
\item How close can $Y$ approach $Z$?
\item Given $\e>0$, what portion of $Y$ enters into the $\epsilon$-neighborhood of $Z$?
\end{enumerate}

It turns out that volumes of compact orbits are the only complexity which measures their quantitative isolation property. The following theorem was proved by Margulis in an unpublished note: 
\begin{thm}[Margulis] \label{thm:lattice-intro} \label{main0} 
Let $\Gamma$ be a cocompact lattice in $G$. 
For every $1/3\le s<1$, the following hold for any compact $H$-orbits $Y\ne Z$  in $X$:
\begin{enumerate}
\item 
\[
d(Y,Z )\gg  \alpha_s^{-4/s} \cdot  \op{Vol}(Y)^{-1/s} \op{Vol}(Z)^{-1/s}  
\]
where $\alpha_s=  (\frac{1}{1-s})^{1/(1-s)} $.
\item For all $0<\epsilon < 1$, \[
m_{Y}(Y\cap B(Z, \epsilon)) \ll  \alpha_s^4 \cdot \epsilon^{s}  \cdot \op{Vol}(Z)
\]
where $m_Y$ denotes the $H$-invariant probability measure on $Y$.
\end{enumerate}
In both statements,  the implied constants depend only on the injectivity radius of $\Gamma\ba G$  (see \eqref{fin0} and \eqref{fin} for more details).
\end{thm}

\begin{Rmk} \rm 
\begin{enumerate}
\item
By  recent works  (\cite{MM}, \cite{BFMS}),  there may be infinitely many compact $H$-orbits only when $\Gamma$ is an arithmetic lattice.

\item  Theorem~\ref{thm:lattice-intro} for {\it some} exponent $s$ is proved in  
\cite[Lemma 10.3]{EMV}. The proof in~\cite{EMV} is based on the effective ergodic theorem
which relies on the arithmeticity of $\Gamma$ via uniform spectral gap on compact $H$-orbits; the exponent $s$ obtained in their approach however is much smaller than $1$.

\item Margulis' proof does not rely on the arithmeticity of $\Gamma$ and is based on the construction of a certain function on $Y$
which measures the distance $d(y, Z)$ for $y\in Y$ (cf. \eqref{mar23}). A similar function appeared first in the
 work of Eskin, Mozes and Margulis in the study of a quantitative version of the Oppenheim conjecture~\cite{EMM-UppB}, and later
 in several other works (e.g., \cite{Eskin-Margulis}, \cite{BQ}, and \cite{EMiMo}).

\end{enumerate}
\end{Rmk}

\medskip

\noindent{\bf General geometrically finite case.} We now consider a general hyperbolic $3$-manifold $\sM=\Gamma\ba \bH^3$.
Denote by $\Lambda\subset \partial \bH^3$ the limit set of $\Gamma$ and by $\core M$ the convex core of $M$, i.e.,
$$\core M=\Gamma\ba \hull \Lambda \subset M$$
where $\hull \Lambda \subset \bH^3$ denotes the convex hull of $\Lambda$. 
In the rest of the introduction, we assume that $M$ is geometrically finite, that is, the unit neighborhood of $\core M$
has finite volume.

 Let $Y\subset X$ be a closed $H$-orbit  and $S_Y=\Delta_Y\ba \bH^2$ be the associated hyperbolic surface,
 where $\Delta_Y<H $ is the stabilizer  in $H$ of a point in $Y$. 
We assume that $Y$ is non-elementary, that is, $\Delta_Y$ is not virtually cyclic; otherwise, we cannot expect an isolation phenomenon for $Y$, as there is
a continuous family of parallel elementary closed $H$-orbits in general when $M$ is of infinite volume.
It is known that $S_Y$ is always geometrically finite \cite[Thm. 4.7]{OS}.

Let $0<\delta(Y)\le 1$ denote the critical exponent of $S_Y$, i.e., the abscissa of the convergence of the series $\sum_{\gamma\in \Delta_Y} e^{-s d(o, \gamma (o))}$ for some $o\in \bH^2$.
We define the following {\it modified critical exponent} of $Y$:
\be  
\ds:=\begin{cases} \delta(Y)&\text{if $S_Y$ has no cusp}\\
 2\delta(Y)-1&\text{otherwise;}\end{cases} 
\ee
note that $0<\ds\le \delta(Y) \le1$, and $\ds=1$ if and only if $S_Y$ has finite area.

In  generalizing Theorem \ref{main0}(1), 
we  first observe that
 the distance  $d(Y, Z)$ between two closed $H$-orbits $Y, Z$ may be zero, e.g., if they both have cusps
going into the same cuspidal end of $X$. To remedy this issue, we use the thick-thin decomposition of $\core M$. For $p\in M$,
we denote by $\inj p$ the injectivity radius at $p$.
For all $\e>0$,  the $\e$-thick part
\be\label{ce} (\core M)_{\e}:=\{p\in \core M: \inj p\ge \e\}\ee is compact, and 
for all sufficiently small $\e>0$, 
the $\e$-thin part given by $\core M-(\core M)_{\e}$
is contained in finitely many disjoint cuspidal ends, i.e., images of  horoballs in $\Gamma\ba\bH^3$.
Let $X_0\subset X $ denote the renormalized frame bundle $\RFM$ (see \eqref{RFM}).
Using the fact that the projection of $X_0$ is contained in $\core M$ under $\pi$,
we define the $\e$-thick part of $X_0$ as follows:
$$X_{\e}:=\{x\in X_0: \pi(x)\in (\core M)_\e \}.$$

The following theorem extends Theorem \ref{main0} to all geometrically finite hyperbolic manifolds:

\begin{thm}\label{thm:isolation}  \label{main} 
Let $\sM$ be a geometrically finite  hyperbolic $3$-manifold.
Let $Y\ne Z$ be non-elementary closed $H$-orbits in $X$, and denote by $m_Y$ the probability Bowen-Margulis-Sullivan measure on $Y$. For every $\frac{\ds}3\le s<\ds$ the following hold.
\begin{enumerate}
\item  For all $0<\e \ll 1$, we have 
\be \label{ine1}
d (Y\cap X_\epsilon, Z)\gg  \alpha_{Y,s}^{-\star/s} \cdot  \left(\frac{v_{Y,\e}}{\tarea Z} \right)^{1/s}
\ee
where \begin{itemize}
\item $v_{Y,\e}=\min_{y\in Y\cap X_\e} {m_Y (B_Y(y, \e))}$ where $B_Y(y, \e)$ is the $\e$-ball around $y$ in the induced metric on $Y$.
\item $\tarea Z$ denotes the tight area of $S_Z$ relative to $M$ (Def. \ref{tight2}).
\item  $\alpha_{Y,s}:=\left(\frac{{\mathsf s}_Y}{ \ds-s}\right) ^{1/ (\ds-s)}$ where $\mathsf s_Y$ is the shadow constant of $Y$ (Def. \ref{p-s}).
\end{itemize}

\item  For all $0<\e \ll 1$, 
\[
m_{Y}(Y\cap B(Z, \e))  \ll \alpha_{Y,s}^\star  \cdot {\e}^{s}\cdot  \tarea Z.
\]
\end{enumerate}
In both statements, the implied constants and $\star$ depend only on $\Gamma$.
\end{thm}

\medskip

\begin{figure}\label{shape}
  \begin{center}
   \includegraphics[width=3.0in]{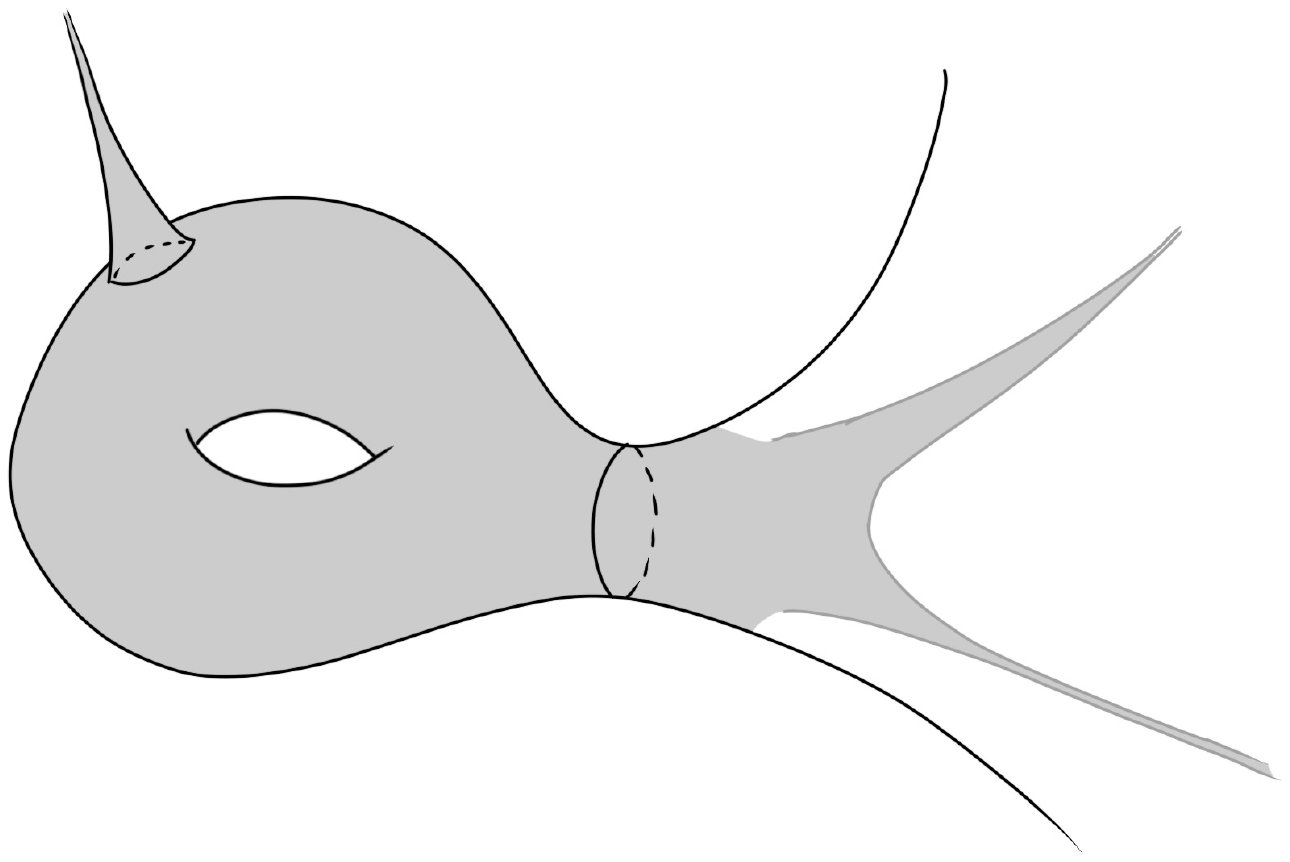} 
 \end{center}
 \caption{ $S\cap \mathcal N(\core M)$}
\end{figure}

\noindent{\bf Remark.}
\begin{enumerate}
    \item 
We give a proof of a more general version of Theorem~\ref{thm:isolation}(1) where $Z$ is allowed to be equal to $Y$ 
(see Corollary \ref{cor:main-dist} for a precise statement).

\item
When $X$ has finite volume,  we have $\ds=1$ and $m_Y$ is $H$-invariant so that $v_{Y, \e}\asymp \e^3 \op{Vol} (Y)^{-1}$. 
Moreover, the tight area $\area_tZ$  and the shadow constant $\mathsf s_Y$
are simply the usual area of $S_Z$ and a fixed constant (in fact, the constant can be taken to be $2$)  respectively. Therefore
 Theorem \ref{main} recovers Theorem \ref{main0}.
Moreover, the {\it exponent} $\star$ depends only on $G$ as well; this follows since the proofs of Theorem \ref{thm:limit-thick} and theorems in Section 10, of which Theorem \ref{main} is a special case, show that $\star$ depends only on $\mathsf s_Y$, $\mathsf p_Y$ and $\delta_Y$, which are all absolute constants in the finite volume case.
\end{enumerate} 

We now give  definitions of  the tight area $\area_t Z$ and  the shadow constant  $\mathsf s_Y$
 for a general geometrically finite case; these are new geometric invariants introduced in this paper.

\begin{Def}[Tight area of $S$] \label{tight2} \rm For a properly immersed geodesic plane $S$ of $M$, 
the {\it tight-area} of $S$ relative to $M$ is given by
$$\tarea(S):=\area (S\cap \mathcal N(\core M)) $$
where $\mathcal N(\core M)=\{p\in M: d(p, q)\le \text{inj}(q)\text{ for some $q\in \core M$} \}$ is the tight neighborhood of $\core M$.
\end{Def}

 We show that $\tarea(S)$ is finite  in Theorem \ref{off1}, by proving that $S\cap \mathcal N(\core M)$ is contained in
the union of a bounded neighborhood of $\core (S)$ and finitely many cusp-like regions (see Fig.\ref{shape}).
We remark that the area of the intersection
$S\cap B(\core M, 1)$ is not finite in general.

\begin{Def}[Shadow constant of $Y$]\label{p-s} \rm For a closed $H$-orbit $Y$ in $X$,
 let $\Lambda_Y\subset \partial \bH^2$ denote the limit set of $\Delta_Y$,
 $\{\nu_{p}:p\in \bH^2\}$  the Patterson-Sullivan density
for $\Delta_Y$, and $B_p(\xi, \e)$  the $\e$-neighborhood of $\xi\in \partial \bH^2$ with respect to  the Gromov metric at $p$. 
The shadow constant of $Y$ is defined as follows:
\be\label{eq:def-b-Y}
\mathsf s_Y:=\sup_{\xi\in \Lambda_Y, p\in [ \xi,\Lambda_Y], 0< \e\le 1/2}
\frac{\nu_{p}(B_p(\xi,\e))^{1/\ds}}{\e\cdot \nu_p(B_p(\xi, 1/2))^{1/\ds}},
\ee
where $[\xi, \Lambda_Y]$ is the union of all geodesics connecting $\xi$ to a point in $\Lambda_Y$.
\end{Def}
We  show that $\mathsf s_Y<\infty$ in Theorem~\ref{spp}.

\begin{rmk}\rm
If $Y$ is convex cocompact, then  for all $0<\e<1$, $v_{Y, \e}\asymp \e^{1+2\ds}$ with the implied constant depending on $Y$.
When $Y$ has a cusp, Sullivan's shadow lemma (cf. Proposition \ref{sq}) implies that
$\lim_{\e\to 0}\frac{\log v_{Y, \e}}{\log \e}$ does not exist.

\end{rmk}

A  hyperbolic $3$-manifold $\sM$ is called {\it convex cocompact {acylindrical}} if $\core M$ is a compact manifold
with no essential discs or cylinders which are not boundary parallel. For such a manifold,
there exists a uniform positive lower bound for $\delta(Y)=\ds$ for all non-elementary closed $H$-orbits $Y$ \cite{MMO};
therefore the dependence of $\ds$ can be removed in Theorem \ref{main} if one is content with taking 
some $s$ which works uniformly for all such orbits.

Examples of $X$ with infinitely many closed $H$-orbits are provided by the following theorem which can be deduced from (\cite{MMO}, \cite{MMO2}, \cite{BO}):
\begin{thm} 
Let $\sM_0$ be an arithmetic hyperbolic $3$-manifold with a properly immersed geodesic plane.
Any geometrically finite acylindrical hyperbolic $3$-manifold $\sM$ which covers $\sM_0$
contains infinitely many non-elementary properly immersed geodesic planes.
\end{thm}

It is easy to construct examples of $\sM$ satisfying the hypothesis of this theorem. For instance,
if $\sM_0$ is an arithmetic hyperbolic $3$-manifold with a properly embedded compact geodesic plane $P$, 
 $\sM_0$ is covered by a geometrically finite acylindrical manifold $\sM$ whose convex core has boundary isometric to $P$.
 
\medskip
\begin{figure}\label{tttt2}
  \begin{center}
   \includegraphics[width=3.0in]{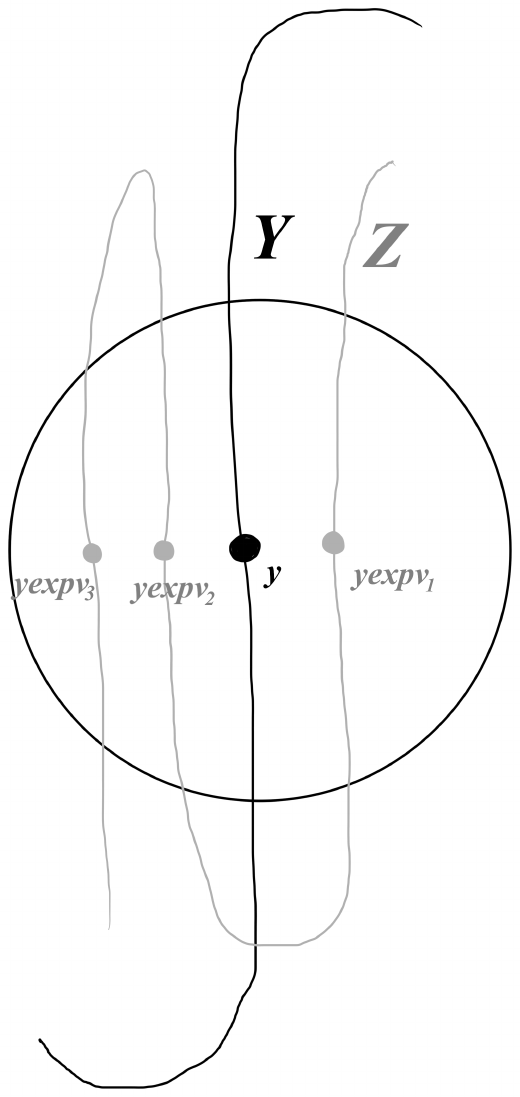} 
   \caption{$I_Z(y)$}
 \end{center}
\end{figure}

Finally, we mention the following application of Theorem \ref{main} in view of recent interests in  related counting problems \cite{CMN}.
\begin{cor} \label{vol1} Let $\Vol(M)<\infty$, and let $\mathcal N(T)$ denote the number of properly immersed totally geodesic  
planes $P$ in $M$ of area at most $T$.
Then for any $1/2<s<1$, we have
$$\mathcal N(T) \ll_sT^{(6/s)-1}\quad \text{ for all $T>1$}; $$
see Corollary \ref{volvol} for a detailed information on the dependence of the implied constant.
\end{cor}
 We remark that when $\Vol (M)<\infty$,
 the heuristics suggest $s={\op{dim} G/ H}=3$ in  Theorem \ref{main} and hence $\mathcal N(T)\ll T$  in Corollary \ref{vol1}.
Indeed, when $\Gamma=\PSL_2 (\mathbb Z[i])$,  the asymptotic $\mathcal N(T)\sim c  \cdot T$, as suggested in ~\cite{Sarnak}, has been obtained by Jung \cite{Jung} based on
subtle number theoretic arguments.

\begin{Rmk}\label{volvol22}\rm  
We can also obtain an estimate for $\mathcal N(T)$ for a general geometrically finite hyperbolic manifold.
By \cite{MMO} and \cite{BO},  if $\Vol (M)=\infty$, there are only finitely many properly immersed geodesic planes of finite area
(note that they are necessarily contained in the convex core of $M$); hence $\sup_{T} \mathcal N(T) <\infty$.
Our methods give that
there exists $N_0\ge 1$ (depending only on $G$) such that for any $1/2<s<1$, we have
$$\mathcal N(T)\ll_s  \Vol (\text{unit-nbd of core} M)\, \e_M^{-N_0} T^{\frac{6}{s}-1}   $$
where the implied constant depends only on $s$ (see Remark \ref{volvol2} for details). Note that this kind of upper bound is meaningful despite the finiteness result mentioned above, as the implied constant is independent of $M$.
\end{Rmk}

\medskip

\noindent{\bf Discussion on proofs.} We discuss some of the main ingredients of the proof of Theorem \ref{main}. 
First consider the case when $X=\Gamma\ba G$ is compact (the account below deviates slightly from Margulis' original argument).
Let $\e_X$ be  the minimum injectivity radius of points in $X$. 
The Lie algebra of $G$ decomposes as $\mathfrak{sl}_2(\mathbb R) \oplus i\mathfrak{sl}_2(\mathbb R)$. Hence,
for each $y\in Y$, the set
$$
I_Z(y):=\{v\in i\mathfrak{sl}_2(\mathbb R): 0<\|v\| <\e_X, \; \;y\exp(v)\in Z\}
$$
keeps track of all points of $Z\cap B(y, \e_X)$  in the direction transversal to $H$ (see Fig. 2).

Therefore, the following function $f_s:Y\to [2,\infty)$ ($0<s<1$) encodes the information on the distance $d(y, Z)$:
\be\label{mar23}
f_s(y)=\begin{cases} \sum_{v\in I_Z(y)}\|v\|^{-s} & \text{if $I_Z(y)\neq\emptyset $}\\
\e_X^{-s}&\text{otherwise}
\end{cases}.
\ee
A function of this type is referred to as a {\em Margulis function} in the literature.

The proof of Theorem \ref{main0} is based on the following fact:
the average of $f_s$ is controlled by the volume of $Z$, i.e., 
\be\label{my} 
m_Y(f_s)\ll_s \op{Vol}(Z).
\ee


We prove the estimate in~\eqref{my} using the following super-harmonicity type inequality: for any $1/3\le s<1$,
there exist $t=t_s>0$ and $b=b_s>1$
such that for all $y\in Y$,
\be\label{ate} {\mathsf A}_t f_s(y) \le \frac{1}{2} f_s(y)+ b\op{Vol}(Z)\ee
where  $({\mathsf A}_t f_s)(y)= \int_0^1 f_s(yu_ra_t)dr$, $u_r=\left(\begin{smallmatrix} 1& 0\\r& 1\end{smallmatrix}\right)$, and $a_t=\left(\begin{smallmatrix} e^{t/2}& 0\\0& e^{-t/2}\end{smallmatrix}\right)$.

The proof of~\eqref{ate} is based on the inequality \eqref{eq:linear-alg}, which is essentially a lemma in linear algebra.
We refer to the Appendix (section \ref{sec:co-cpct-case}), where  a   more or less
complete proof of Theorem \ref{main0} is given.

 For a general geometrically finite hyperbolic manifold, many changes are required, and several technical difficulties arise.
In general, there is no positive lower bound for the injectivity radius on $X$, and
the shadow constant of $Y$ appears in the linear algebra lemma (Lemma~\ref{lem:vect-contr}).
These facts force us  to incorporate the height of  $y$ as well as the shadow constant of $Y$ in the definition of the Margulis function 
(see Def.\ \ref{MFF}).  The correct substitutes for the volume measures on $Y$ and $Z$ turn out to be 
  the Bowen-Margulis-Sullivan probability measure $m_Y$  and  the tight area of $Z$ respectively.  
  
It is more common in the existing literature on the subject to define the operator ${\mathsf A}_t$ using averages over large spheres in $\mathbb H^2$. Our operator ${\mathsf A}_t$ however is defined using averages over expanding horocyclic pieces; this choice is more amenable to the change of variables and  iteration arguments for Patterson-Sullivan measures.
Indeed, for a locally bounded Borel function $f$ on $Y\cap X_0$ and for any $y\in Y\cap X_0$,
$$({\mathsf A}_t f)(y)= \frac{1}{\mu_y([-1,1])}\int_{-1}^1 f(yu_ra_t)d\mu_y(r)$$ where $\mu_y$ is the Patterson-Sullivan measure on $yU$ (see \eqref{measure})

When $X$ is compact and hence $m_Y$ is $H$-invariant,~\eqref{my} follows by simply integrating~\eqref{ate} with respect to $m_Y$. In general, we resort to Lemma~\ref{lem:F-integ} the proof of which is based on an iterated version of $\eqref{ate}$ for ${\mathsf A}_{nt_0}$, $n\in \N$, for some $t_0>0$ as well as on the fact that the Bowen-Margulis-Sullivan  measure $m_Y$ is $a_{t_0}$-ergodic.

In fact, the main technical result of this paper can be summarized as follows:
\begin{prop} \label{prop:main-tech}
Let $\Gamma$ be a geometrically finite subgroup of $G$.
Let $Y\ne Z$ be non-elementary closed $H$-orbits in $X=\Gamma\ba G$, and set $Y_0:=Y\cap X_0$.  For any $\frac{ \ds}{3} \le s<\ds$, there exist $t_s>0$ and
a locally bounded  Borel function $F_s:Y_0\to (0, \infty)$  with the following properties:
\begin{enumerate}
\item For all $y\in Y_0$,
$$d(y, Z) ^{-s}\le \mathsf s_Y^{\star} F_s(y).$$
\item
For all $y\in Y_0$ and $n\ge 1$, 
\[
\left({\mathsf A}_{nt_s}  F_s\right) (y)\le \frac{1}{2^{n}} F_{s}(y)+  \alpha_{Y,s}^\star  \area_t(S_Z).
\]

\item
There exists  $1< \sigma\ll {\mathsf s}_Y^{\star}$
such that  for all $y\in Y_0$ and for  all $h\in H$ with $\|h\|\ge 2$ and $yh\in Y_0$, 
\[
\sigma^{-1} F_{s} (y) \le F_{s}(yh) \le \sigma F_{s}(y).
\]

\end{enumerate}
\end{prop}

Finally we mention that the reason that
we can take the exponent $s$ arbitrarily close to $\ds$ lies in the two ingredients of our proof: firstly,
 the linear algebra lemma  (Lemma~\ref{lem:vect-contr})  is obtained for all $\ds/3\le s<\ds$ and secondly, for any $y\in Y\cap X_0$, we can find
$|r|<1$ so that $yu_r\in X_0$ and the height of $yu_r$ can be lowered to be $O(1)$ by the geodesic flow of time comparable to
the logarithmic height of $y$; see Lemma \ref{lem:one-return} for the precise statement.

\medskip

\noindent{\bf Organization.} We end this introduction with an outline of the paper. 
In \S\ref{sec:notation}, we fix some notation and conventions to be used throughout the paper.  In \S\ref{sec:ta}, we show the finiteness of the tight area of a properly immersed geodesic plane. In \S\ref{sec:closed}, we show the finiteness of the shadow constant of a closed $H$-orbit.  In \S\ref{sec:linear-alg}, we prove a lemma from linear algebra; this lemma is a key ingredient to prove a local version of our main inequality.
\S\ref{sec:height-func} is devoted to the study of the height function in $X_0$. In \S\ref{AVE}, the definition of the Markov operator and a basic property of this operator are discussed. In \S\ref{RET}, we prove the return lemma, and use it to
obtain a uniform control on the number of sheets of $Z$ in a neighborhood of $y$.
In \S\ref{sec:isolation}, we  construct the desired Margulis function and prove the main inequalities.
 In \S\ref{final2}, we give a proof of Theorem \ref{main}.
In the Appendix (\S\ref{sec:co-cpct-case}),
we provide a  proof of Theorem \ref{main0}.

\medskip
\noindent
{\bf Acknowledgement.}\ A.M.\ would like to thank the Institute for Advanced Study for its hospitality during the fall of 2019 where part this project was carried out. We would like to thank the referee for a careful reading of our paper and for making many useful comments.

\section{Notation and preliminaries}\label{sec:notation}
In this section, we review some definitions and introduce notation which will be used throughout the paper.

We set $G=\PSL_2(\c)\simeq \op{Isom}^+(\bH^3)$, and $H=\PSL_2(\br)$.
We fix $\bH^2\subset \bH^3$ with an orientation so that $\{g\in G: g(\bH^2)=\bH^2\}=H$.
Let $A$ denote the following one-parameter subgroup of $G$:
  $$A=\left\{ a_t=\begin{pmatrix} e^{t/2} & 0 \\ 0 & e^{-t/2} \end{pmatrix} : t\in \br \right \}.$$

Set $K_0=\op{PSU}(2)$ and $M_0$ the centralizer of $A$ in $K_0$.
We fix a point $o\in \bH^2\subset \bH^3$ and a unit tangent vector $v_{o}\in \op{T}_{o}(\bH^3)$ so that
their stabilizer subgroups are $K_0$ and $M_0$ respectively.  
The isometric action of $G$ on $\bH^3$ induces identifications $G/K_0=\bH^3$, $G/M_0=\T^1 \bH^3$, and $G=\op{F}\bH^3$ where $\T^1 \bH^3 $ and $\F \bH^3$ denote, respectively, the unit tangent bundle and the oriented frame bundle
over $\bH^3$. Note also that $H\cap K_0=\op{PSO(2)}$ and that  $H(o)= \bH^2$.

The right translation action of $A$ on $G$ induces the geodesic/frame flow on $\T^1 \bH^3$ and $\F \bH^3 $, respectively. Let $v_{o}^{\pm}\in \partial \bH^3$ denote the forward and backward end points of the geodesic given by $v_{o}$. For $g\in G$, we define
$$g^{\pm}:=g(v_{o}^{\pm})\in \partial \bH^3 .$$

Let $\Gamma <G$ be a discrete torsion-free subgroup. 
We set $$\sM:=\Gamma\ba \bH^3\quad \text{ and } \quad X:=\Gamma\ba G\simeq \F M.$$ 
We denote by $\pi: X\to M$ the base point projection map.
Denote by $\Lambda=\Lambda(\Gamma)$ the limit set of $\Gamma$. The convex core of $M$ is given by $\core M=\Gamma\ba \text{hull} (\Lambda)$.
Let  $X_0$ denote the renormalized frame bundle $\RFM$, i.e.,
 \be\label{RFM} X_0=\{[g]\in X: g^{\pm}\in \La\} ,\ee
that is, $X_0$ is the union of all the $A$-orbits whose projections to $M$ stay inside $\core M$.
We remark that $X_0$ does not surject onto $\core M$ in general.

In the whole paper, we assume that $\Gamma$ is geometrically finite, that is, the unit neighborhood of $\core M$ has finite volume.
This is equivalent to the condition that
 $\Lambda$ is the union of the radial limit points and bounded parabolic limit points: $\La=\La_{rad}\bigcup\La_{bp}$ (cf. \cite{Bow}, \cite{MT}).
A point $\xi\in \Lambda$ is called {\it radial} if the projection of a geodesic ray toward to $\xi$ accumulates on $M=\Gamma\ba \bH^3$,
{\it parabolic} if it is fixed by a parabolic element of $\Gamma$, and {\it bounded parabolic}
if it is parabolic and $\op{Stab}_\Gamma(\xi)$ acts co-compactly on $\Lambda-\{\xi\}$.
In particular, for $\Gamma$ geometrically finite, the set of parabolic limit points $\Lambda_p$ is equal to $\Lambda_{bp}$.
 For $\xi\in \Lambda_{p}$, the rank of the free abelian subgroup $\op{Stab}_\Gamma(\xi)$ is referred to as the rank of $\xi$.

A geometrically finite group $\Gamma$ is called {\it convex cocompact} if $\core M$ is compact, or equivalently, if $\La=\Lambda_{rad}$.

We denote by $N$ the expanding horospherical subgroup of $G$ for the action of $A$:
 $$N =\left\{u_s=\begin{pmatrix} 1 & 0 \\ s & 1\end{pmatrix}: s\in \c \right\}.$$ 
 For $\xi\in \La_p$,  a horoball $\tilde{\mathfrak h}_{\xi}\subset G$ based at $\xi$ is of the form
\be\label{hoho}
\tilde{\mathfrak h}_\xi(T) =g N { A}_{( -\infty, -T]} K_0 \;\;\text{for some $T\ge 1$}
\ee
where $g\in G$ is such that $g^-=\xi$ and ${A}_{(-\infty,-T]}=\{a_{t}: -\infty<t\le -T\}$.
Its image  $\tilde{\mathfrak h}_\xi(o)$  in $\bH^3$  is called a horoball in $\bH^3$ based at $\xi$.
By a horoball $\mathfrak h_\xi$ in $X$ and in $M$,
 we mean  their respective images of horoballs $\tilde{\mathfrak h}_{\xi}$ and $\tilde{\mathfrak h}_\xi(o)$
 in  $X$  and $M$  under the  corresponding projection maps.

\subsection*{Thick-thin decomposition of $X_0$}
We fix a Riemannian metric $d$ on $G$ which induces the hyperbolic metric on $\bH^3$.
By abuse of notation, we use $d$ to denote the distance function on $X$ induced by $d$, as well as on $M$.
For a subset $S\subset \spadesuit$ and $\e>0$, $B_\spadesuit(S, \e)$ denotes the set $\{x\in \spadesuit: d(x, S)\le \e\}$.
When $\spadesuit$ is a subgroup of $G$ and $S=\{e\}$, we simply write $B_\spadesuit(\e)$ instead of $B_\spadesuit(S, \e)$.
When there is no room for confusion for the ambient space $\spadesuit$, we omit the subscript $\spadesuit$.

For $p\in M$, we denote by $\op{inj} p$ the injectivity radius at $p\in M$, that is: the supremum $r>0$ such that
the projection map $\bH^3\to  M=\Gamma\ba \bH^3$ is injective on the ball $B_{\bH^3}(\tilde p, r)$ where $\tilde p\in \bH^3$
is such that $p=[\tilde p]=\tilde p\Gamma$.
For $S\subset M$ and $\e>0$, we call the subsets $\{p\in S: \inj (p)\ge \e\}$ and
$\{p\in S: \inj (p)< \e\}$ the $\e$-thick part and the $\e$-thin part of $S$ respectively.

As $M$ is geometrically finite,  
 $\core M$ is contained in a union of its $\e$-thick part $(\core M)_\e$ and finitely many disjoint horoballs for all small $\e>0$ (cf. \cite{MT}).
 If $p=gu_sa_{-t}o$ is contained in a horoball $ \mathfrak h_\xi= g N {A}_{( -\infty, -T]}(o)$,
 then $\inj (p)\asymp e^{-t}$ for all $t\gg T$, this is a standard fact see, e.g.,~\cite[Prop.~5.1]{KO}.

Let $\e_M>0$ be the supremum of $\e$ with respect to which such a decomposition of $\core M$ holds.
 We call the $\e_M$-thick part of $\core M$ the compact core of $M$, and
  denote by $M_{{\rm cpt}}$.

For $x=[g]\in X$,
we denote by $\op{inj}(x)$ the injectivity radius of $\pi(x)\in M$.
For  $\epsilon>0$, we set 
$$X_\epsilon:=\{x\in X_0:\inj (x) \ge \e\}.$$ 

{ We set $\e_X=\e_M/2$};  note that $X_0-X_{\e_X}$ is either empty or is contained in a union of horoballs in $X$. 

\subsection*{Convention} By an absolute constant, 
we mean a constant which depends at most on $G$ and $\Gamma$. 
We will use the notation $A\asymp B$ when the ratio between the two lies in $[C^{-1}, C]$
for some absolute constant $C\ge 1$. We write $A\ll B^\star$ (resp. $A\asymp B^\star$,  $A\ll \star B$) to mean that $A\le C B^L$ (resp. 
$C^{-1}B^L\le A\le C B^L$, $A\le C\cdot B$)
for some absolute constants $C>0$ and $L>0$. 

\section{Tight area of a properly immersed geodesic plane}\label{sec:ta}
In this section, we show that the tight area of a properly immersed geodesic plane of $M$ is finite.

{For a closed subset $Q\subset M$,
we define {\it the tight neighborhood of $Q$} by
$$\mathcal N(Q):=\{p\in M: d(p, q)\le \text{inj}(q)\text{ for some $q\in Q$} \}.$$

We are mainly interested in the tight neighborhood of $\core M$.
If $M$ is convex cocompact, $\mathcal N(\core M)$ is compact.  
 In order to describe the shape of $\mathcal N(\core M)$ in the presence of cusps, 
  fix a set $\xi_1, \cdots, \xi_\ell$
 of $\Gamma$-representatives of $\Lambda_{p}$, cf.~\cite{MT}.
Then $\core M$ is contained in the union of $\Mt $ and a disjoint union $\bigcup \mathfrak h_{\xi_i}$ of horoballs based at the $\xi_i$s. 
  
 Consider the upper half-space model $\bH^3=\{(x_1, x_2,y): y>0\}=\br^2\times \br_{>0}$, and
let  $\infty\in \Lambda_{p}$.
Let $p:\bH^3\to M$ denote the canonical projection map. 
 As $\infty$ is a bounded parabolic fixed point, there exists a bounded rectangle, say, $I\subset \br^2$ and $r>0$  (depending on $\infty$)
 such that \begin{enumerate}
 \item $p(I\times \{y> r \})\supset \mathcal N(\mathfrak h_\infty\cap \core M)$ and
 \item  $p( I\times \{r\})\subset B(\Mt, R)$ 
 \end{enumerate} where $R$ depends only on $M$.
 We call this set $\mathfrak C_\infty:=I\times \{y\ge  r \}$ a chimney for $\infty$ (cf. Figure 3).  

\begin{figure}
  \begin{center}
   \includegraphics[width=2.2in]{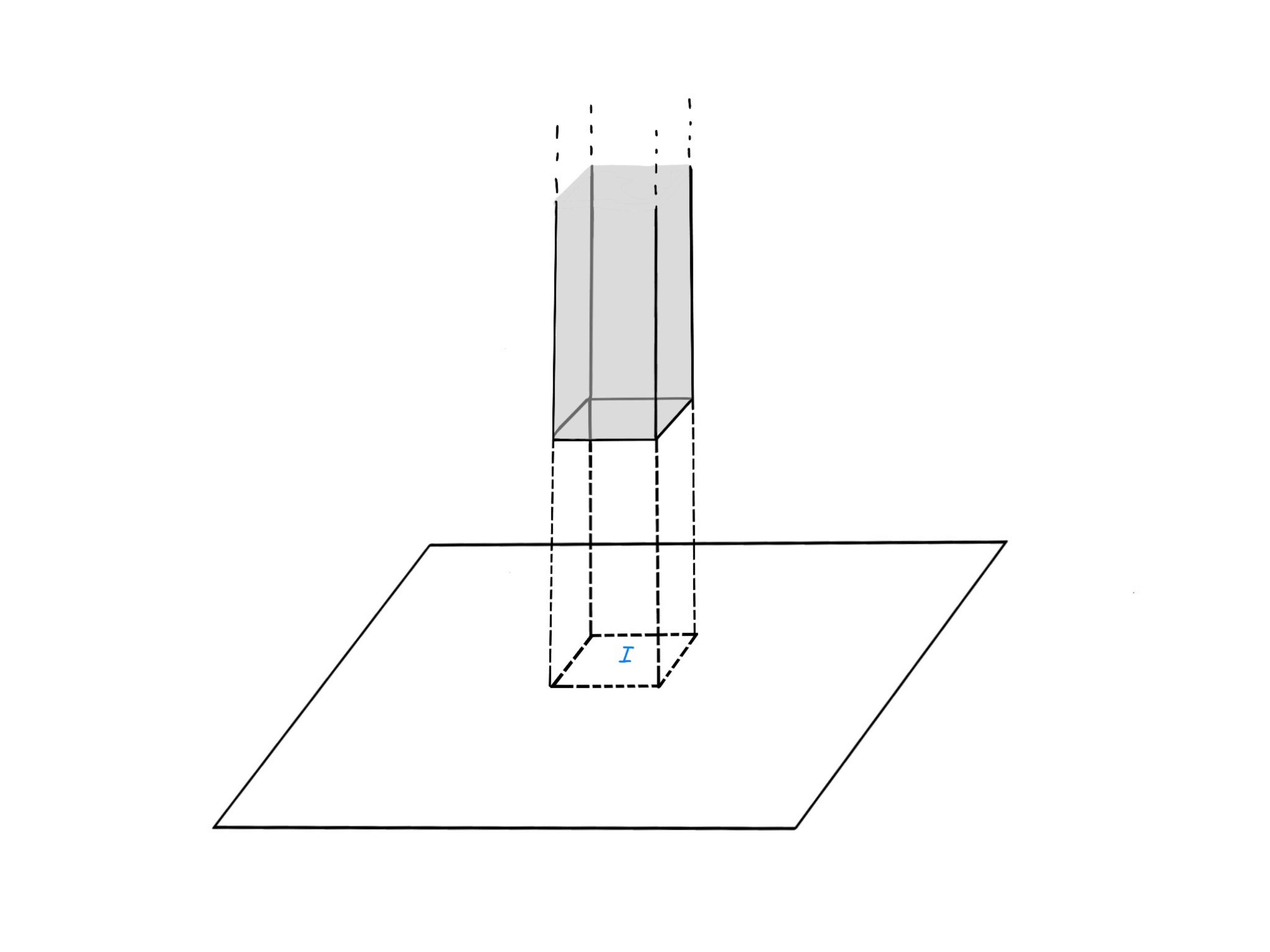} 
 \end{center}
 \caption{Chimney}
\end{figure}

Note that increasing $R$ if necessary, we have  
\be\label{chin} \mathcal N(\core M)\subset B(\Mt, R) \cup \biggl(\bigcup_{1\le i\le \ell} p(\mathfrak C_{\xi_i})\biggr)\ee
where $\mathfrak C_{\xi_i}$ is a chimney for $\xi_i$.

 \begin{Def}\rm For a properly immersed geodesic plane $S$ of $M$, we define
the {\it tight-area} of $S$ relative to $M$ as follows:
$$\tarea(S):=\area (S\cap \mathcal N(\core M)).$$
\end{Def}

\begin{thm}\label{lem:VZ-finite1}\label{off1}
For a properly immersed non-elementary geodesic plane $S$ of $M$, 
 we have 
 $$1\ll \tarea(S) <\infty$$
 where the implied multiplicative constant depends only on $M$.
\end{thm}

\begin{proof}

Since no horoball can contain a complete geodesic, it follows that $S$ intersects the compact core $M_{{\rm cpt}}$. 
Therefore, 
\[
\tarea S\ge 4\pi\sinh^2(\e_X/2),
\]
as $S\cap M_{\rm cpt}$ contains a hyperbolic disk of radius $\e_X$ (see~ Section \ref{sec:notation}).  
This implies the lower bound.

We now turn to the proof of the upper bound.  We use the notation in \eqref{chin}.
 Fix a geodesic plane $P\subset \mathbb H^3$ which covers $S$ and let $\Delta={\rm Stab}_\Gamma(P)$. 
Fix a Dirichlet domain $D$ in $P$ for the action of $\Delta$. As $\Delta\ba P$ is geometrically finite,
the Dirichlet domain is a finite sided polygon; hence, $D\cap \hull (\Delta)$ has finite area, and the set $D-\hull (\Delta)$ is a disjoint union of finitely many flares, where a flare is a region bounded by three geodesics as shown in Figure 4.
Fixing a flare $F\subset D-\hull (\Delta)$, it suffices to show that
$\{x\in F: p(x)\in \mathcal N (\core M)\}$ has finite area. 
As $S$ is properly immersed, the set $\{x\in F: d(p(x), \Mt)\le R\}$ is bounded.
Therefore, fixing a chimney $\mathfrak C_{\xi_i}$ as above, it suffices to show that
the set $\{x\in F: p(x)\in \mathfrak C_{\xi_i}\}=F\cap \Gamma \mathfrak C_{\xi_i}$ has finite area. 

Without loss of generality, we may assume $\xi_i=\infty$. 
We will denote by $\partial F$
the intersection of the closure of $F$ and $\partial P$, and let $F_\e\subset  \overline F$ denote 
the $\e$-neighborhood of $\partial F$ in the Euclidean metric
in the unit disc model of $\overline P$ (cf.\ Figure 4).

Fix $\e_0>0$ so that \be\label{fe} F_{\e_0}\cap \{x\in D: d(p(x), \Mt)<R\}=\emptyset;\ee  such $\e_0$ exists, as $S$ is a proper immersion.
 Writing  $\mathfrak C_\infty=I\times \{y\ge r\}$ as above, let $H_\infty:=\br^2\times \{y>r\}$, and set $\Gamma_\infty:=\op{Stab}_\Gamma (\infty)$.

We claim that  
\be\label{wq} 
\#\{\gamma H_\infty: F_{\e_0/2}\cap \gamma \mathfrak C_\infty\ne \emptyset\}<\infty.
\ee
Suppose not.  Since $\Gamma H_\infty$ is closed in the space of all horoballs in $\bH^3$,
there exists a sequence of distinct $\gamma_i(\infty)\in \Gamma(\infty)$ 
such that $F_{\e_0/2}\cap \gamma_i \mathfrak C_\infty\ne \emptyset$ and  the size of the horoballs $\gamma_i H_\infty$ goes to $0$ in the Euclidean metric in the ball model of $\bH^3$.
Note that  if $\infty$ has rank $2$, then  $\Gamma_\infty(I\times \{r\})=\br^2\times \{r\}$ and that
  if $\infty$ has rank $1$, then $\Gamma_\infty(I\times \{r\})$ contains a region between two parallel horocycles  in $\br^2\times \{r\}$.
Since  $P\cap \gamma_i \mathfrak C_\infty\ne \emptyset$, it follows 
 that $ P\cap \gamma_i (\Gamma_\infty (I \times \{r\})) \ne \emptyset$. Moreover, if $i$ is large enough so that
 the Euclidean size of $\gamma_i H_\infty$ is smaller than $\e_0/2$, 
 the condition $F_{\e_0/2} \cap \gamma_i \mathfrak C_\infty\ne \emptyset$ implies that
 $ F_{\e_0} \cap \gamma_i (\Gamma_\infty (I \times \{r\})) \ne \emptyset$.
  This yields  a contradiction to \eqref{fe} since $p(I\times \{r\})$
is contained in the $R$-neighborhood of $\Mt$, proving the claim.

By Claim \ref{wq}, it is now enough to show that, fixing a horoball $\gamma H_\infty$,
 the intersection $F_{\e_0}\cap \gamma\Gamma_\infty \mathfrak C_\infty$ has finite area.
 Suppose that   $F_{\e_0}\cap \gamma\Gamma_\infty \mathfrak C_\infty$  is unbounded in $P$; otherwise the claim is clear.
Without loss of generality, we may assume $\gamma=e$, by replacing $P$ by $\gamma^{-1}P$ if necessary.
If $\infty\notin \partial P$, then  $F_{\e_0}\cap \Gamma_\infty \mathfrak C_\infty$, being contained in $ P\cap H_\infty$, is a bounded subset of $P$; contradiction.
 Therefore,
 $\infty\in \partial P$. Then, as $F_{\e_0}\cap \Gamma_\infty \mathfrak C_\infty\subset F_{\e_0}\cap H_\infty$ is unbounded, we have $\infty\in \partial F$.
Since $F$ is a flare, it follows that $\infty$ is not a limit point for $\Delta$. This implies that the rank of $\infty$ in $\Lambda_{p}$ is $1$ \cite[Lem. 6.2]{OS}.
 Therefore $\Gamma_\infty \mathfrak C_\infty$ is contained in a subset of the form $T\times \{y\ge r\}$ where $T$ is a strip between two parallel lines $L_1, L_2$
  in $\br^2$.
  Since $\infty$ is not a limit point for $\Delta$, the vertical plane $P$ is not parallel to the $L_i$.
  Therefore, the intersection $F_{\e_0}\cap \Gamma_\infty \mathfrak C_\infty$, being a subset of $  P\cap (T\times \{y\ge r\})$,  is contained in a cusp-like region, isometric to $\{(x,y)\in \bH^2: y\ge r\}$ and $x$ is also bounded from above and below (recall that $P$ is not parallel to the $L_i$). This finishes the proof.
\end{proof}

\begin{figure}
  \begin{center}
  \includegraphics[width=2.0in]{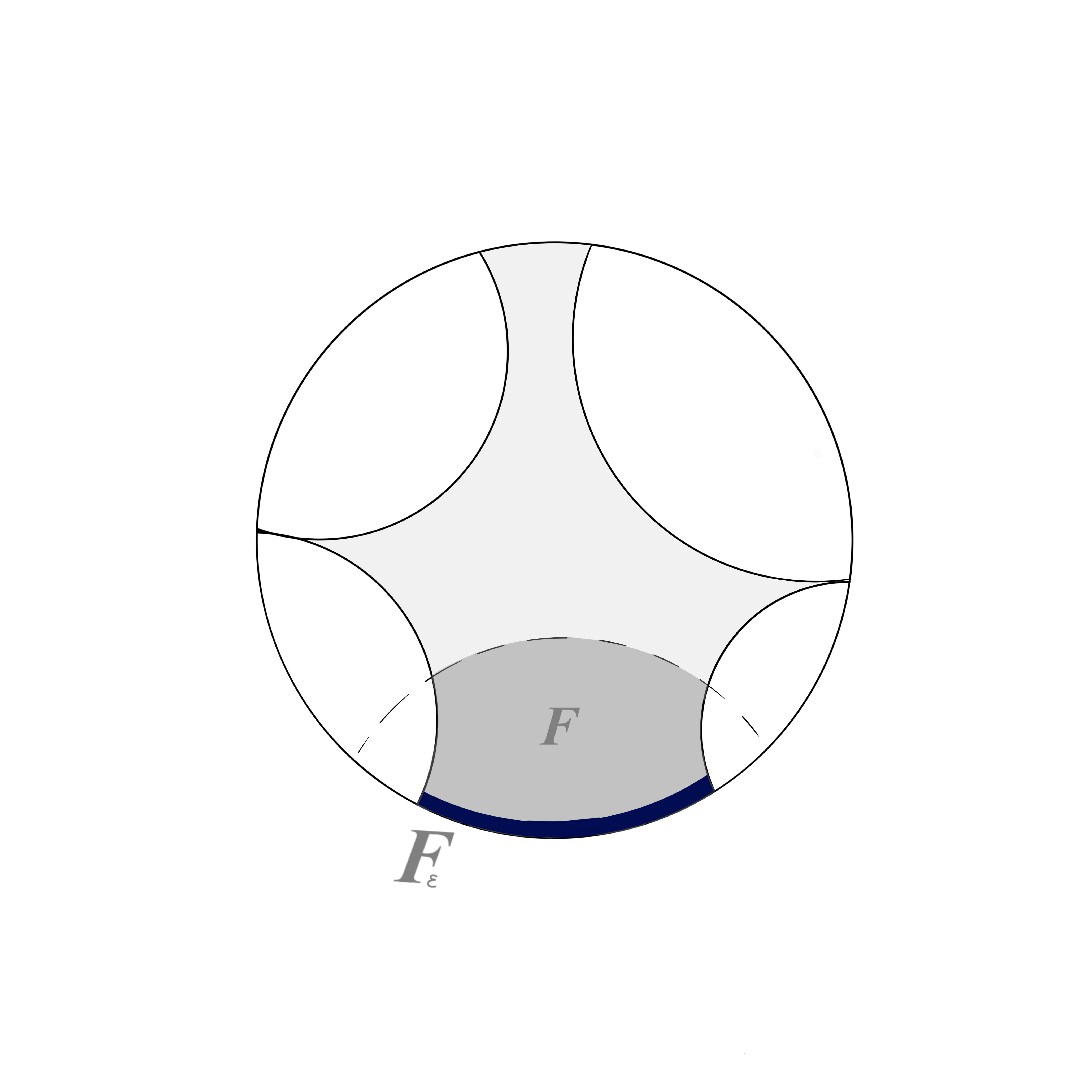} 
 \end{center}
 \caption{Flare $F$ and $F_\e$ }
\end{figure}

The proof of the above theorem demonstrates  that the portion of  $S$, especially of the flares of $S$, staying in the {\it tight}
 neighborhood of $\core M$ can go to infinity
 only in cusp-like shapes, by visiting the chimneys of horoballs of $\core M$ (Fig. \ref{shape}). 
 This is not true any more if we replace the tight neighborhood of $\core M$ by
 the unit neighborhood of $\core M$. More precisely if $\Lambda$ contains a parabolic limit point  of rank one which is not stabilized by any element of 
 $\pi_1(S)$, then some  region of $S$ with infinite area can stay inside the unit neighborhood of $\core M$.
This situation may be compared to the presence of  divergent geodesics in finite area setting.

\section{Shadow constants}\label{sec:closed}
In this section, fixing a closed non-elementary $H$-orbit $Y$ in $X$, we recall the definition of Patterson-Sullivan measures $\mu_y$
on horocycles in $Y$, and relate its density with the shadow constant $\mathsf s_Y$, which we show  is a finite number.

Set $  \Delta_Y:=\op{Stab}_H(y_0)$ to be the stabilizer of a point $y_0\in Y$;  note that despite the notation, $\Delta_Y$ is uniquely determined up to a conjugation by an element of $H$.
As $\Gamma$ is geometrically finite and $Y=Hy_0$ is a closed orbit, 
the subgroup $\Delta_Y$ is a geometrically finite  subgroup of $H$, \cite[Thm.~4.7]{OS}.
 We denote by $\Lambda_Y\subset \partial \bH^2 $ the limit set of $\Delta_Y$.
 Let $0<\delta(Y)\le 1$ denote the critical exponent of $\Delta_Y$, or
 equivalently, the Hausdorff dimension of $\La_Y$.

 We denote by $\{\nu_p=\nu_{Y, p}: p\in \bH^2\}$ the Patterson-Sullivan density  for $\Delta_Y$, normalized so
  that $|\nu_o|=1$. This means that the collection $\{\nu_p\}$ consists of Borel measures
 on $\Lambda_Y$ satisfying that
for all $\gamma\in \Delta_Y$, $p, q \in \bH^2$, $\xi\in \La_Y$,
 $$\frac{d \gamma_* \nu_p}{d\nu_p}(\xi) =e^{-\delta(Y) \beta_\xi(\gamma^{-1}(p), p)}\quad \text{ and }\quad
 \frac{d \nu_q}{d\nu_p}(\xi) =e^{-\delta(Y) \beta_\xi(q, p)}$$
 where $\beta_\xi(\cdot, \cdot)$ denotes the Busemann function. In the sequel we will refer to the first identity above as $\Gamma$-conformality of $\{\nu_p\}$.
 
 As $\Delta_Y$ is geometrically finite, there exists a unique Patterson-Sullivan density up to a constant multiple.

\subsection*{PS-measures on $U$-orbits}
  
Set
 $$U :=\left\{u_r=\begin{pmatrix} 1 & 0 \\ r & 1\end{pmatrix}: r\in \br \right\}=N\cap H$$
 which is the expanding horocylic subgroup of $H$.
Using the parametrization ${r}\mapsto u_{r}$, we may identify $U$ with $\br$. 
Note that for all $r,t\in \br$,
\[
a_{-t}u_{r} a_t=u_{e^t{r}}.
\]

For any $h\in H$, the restriction of the visual map $g\mapsto g^+$ is 
 a diffeomorphism between $hU$ and $\partial \bH^2-\{h^-\}$.
Using this diffeomorphism, we can define a measure $\mu_{hU}$ on $hU$:
\be\label{eq:def-ps-meas}
d \mu_{hU} (hu_{r} )= e^{\delta(Y) \beta_{(hu_{{r}})^+}(p,hu_{{r}}(p))} d{\nu_{p}}(hu_{{r}})^+ ;
\ee
this is independent of the choice of $p\in \bH^2$. We simply write $d\mu_{h}({r})$ for $d \mu_{hU} (hu_{r} )$.
Note that these measures depend on the $U$-orbits but not on the individual points.
By the {$\Delta_Y$}-invariance and the conformal property of the PS-density, 
we have
 \be\label{measure}d\mu_h (\cal O)=d\mu_{\gamma h} (\cal O)\ee
for any $\gamma\in \Delta_Y$ and 
for any bounded Borel set $\cal O\subset \br$;
therefore $\mu_{y}(\cal O)$ is well-defined for $y\in \Delta_Y\ba H$.

For any $y\in \Delta_Y\ba H$ and any $t\in \br$, we have:
\be\label{p1}
\mu_y([-e^t,e^t])=e^{\delta(Y) t}\mu_{ya_{-t}}([-1,1]).
\ee
 
 Set 
 \be\label{yay} Y_0:=\{[h]\in \Delta_Y\ba H: h^{\pm}\in \Lambda_Y\}\ee
 where $h^{\pm}= \lim_{t\to \pm \infty} ha_t(o) $.

\subsection*{Shadow constant}
As in the introduction, we define the modified critical exponent of $Y$:
\begin{equation} \label{eq: def delta(Y)}
\ds= \begin{cases} 
\delta(Y) &\text{if $Y$ is convex cocompact}
\\ 2\delta(Y)-1&\text{otherwise.}\end{cases}
\end{equation}
If $Y$ has a cusp, then $\delta(Y)>1/2$, and hence
$0<\ds \le \delta(Y)\le  1$.

Define 
\be\label{sc}
{\mathsf p}_Y=\sup_{y\in Y_0, 0< r\le 2} \frac{\mu_{y}([-r,r])^{1/\ds}}{r\cdot \mu_{y}([-1,1])^{1/\ds}};
\ee
the range $0<r\leq 2$ is motivated by our applications later, see e.g.,~\eqref{eq:local-comp-omega-2}.

Recall the shadow constant $\mathsf s_Y=\sup_{0<\e\le 1/2} \mathsf s_Y(\e)$ in \eqref{p-s} where 
\be\label{eq:def-b-Y2}
\mathsf s_Y(\e):= \sup_{\xi\in \Lambda_Y, p\in [\xi,\La_Y]}
\frac{\nu_{p}(B_p(\xi,\e))^{1/\ds}}{\e\cdot \nu_p(B_p(\xi, 1/2))^{1/\ds}}.
\ee
where $[\xi, \Lambda_Y]$ is the union of all geodesics connecting $\xi$ to a point in $\Lambda_Y$, and $B_p(\xi,\cdot)$ is as in~\eqref{bp}. 

The rest of this section is devoted to the proof of the following theorem using a uniform version of Sullivan's shadow lemma.
\begin{thm}\label{spp}
We have $$\mathsf s_Y\asymp \mathsf p_Y <\infty.$$
\end{thm}

In principle, this definition of $\mathsf s_Y$ involves making a  choice
of $\Delta_Y=\op{Stab}_H(y_0)$, i.e., the choice of $y_0\in Y$, as $\La_Y$ is the limit set of $\Delta_Y$.
However we observe the following:

\begin{lem}\label{lem:pY-indep-y} The constant $\mathsf s_Y$ is independent of the choice of $y_0\in Y$.
\end{lem}
\begin{proof} 
Let $y=y_0h^{-1}\in Y$ for $h\in H$.  Define  $\mathsf s'_Y$ similar to $\mathsf s_Y$  using $\Delta_Y'={\rm Stab}_H(y)=h\Delta_Yh^{-1}$ and put $\nu'_p:=h_*\nu_{h^{-1}p}$ for each $p\in \bH^2$. If  $\xi\in\Lambda_Y$, then 
\begin{align*}
&\frac{d \Bigl((h\gamma h^{-1})_* \nu_p'\Bigr)}{d\nu_p'}(h\xi)=\frac{d \Bigl((h\gamma)_* \nu_{h^{-1}p}\Bigr)}{dh_*\nu_{h^{-1}p}}(h\xi)
=\frac{d \gamma_* \nu_{h^{-1}p}}{d\nu_{h^{-1}p}}(\xi)\\ &=e^{-\delta(Y) \beta_\xi(\gamma^{-1}(h^{-1}p), h^{-1}p)}
=e^{-\delta(Y) \beta_{h\xi}(h\gamma^{-1}h^{-1}(p), p)}.
\end{align*}
Since the limit set of $\Delta'_Y$ is given by $h\Lambda_Y$, this implies that the family $\{\nu'_p: p\in \bH^2\}$ is the Patterson-Sullivan density for $\Delta'_Y$. 
Now for any $0<\e\leq 1$ and $\xi\in\Lambda_Y$, we have 
\[
\nu_{hp}'(B_{hp}(h\xi,\e))=h_*\nu_{p}(B_{hp}(h\xi,\e))=\nu_{p}(h^{-1}B_{hp}(h\xi,\e))=\nu_p(B_p(\xi, \e)).
\]
It follows that $\mathsf s_Y=\mathsf s'_Y$.
\end{proof}

 \subsection*{Shadow lemma}
 Consider the associated hyperbolic plane and its convex core:
 $$S_Y:= \Delta_Y\ba \bH^2\quad \text{ and}\quad  \core (S_Y):=\Delta_Y\ba \hull(\La_Y).$$ 
We denote by $C_Y$ the compact core of $S_Y$, defined as the minimal connected surface whose complement in
 $\op{core}(S_Y)$
is a union of disjoint cusps. If $S_Y$ is convex cocompact, then $C_Y=S_Y$. Let  $$d_Y:=\max\{1, \op{diam} (C_Y)\}.$$

We can write $\core(S_Y)$ as the disjoint union of the compact core $C_0:=C_Y$ and finitely many cusps, say, $C_1, \ldots, C_m$. 
Fix a Dirichlet domain $\cal F_Y\subset \bH^2$ for $\Delta_Y$ containing  the base point $o$.
For each $C_i$,  $0\le i\le m$,  choose the lift $\tilde C_i\subset \Cal F_Y\cap \text{hull}(\Lambda_Y)$ so that
 $\Delta_Y\ba \Delta_Y \tilde C_i=C_i$.  In particular,
$\partial \tilde C_0$ intersects $\tilde C_i$ in an interval for $i\ge 1$. Let $\xi_i\in \La_Y$
be the base point of the horodisc $\tilde C_i$, i.e., $\xi_i=\partial \tilde C_i \cap \partial \bH^2$.
Let $F_{\xi_i}\subset \partial \bH^2-\{\xi_i\}$ be a minimal closed interval so that
 $\Lambda_Y-\{\xi_i\}\subset \op{Stab}_{\Delta_Y}(\xi_i) F_{\xi_i}$. 

For $p\in {}{\mathbb H}^2$, let $d_p$ denote the Gromov distance on $\partial \bH^2$: for $\xi\ne \eta\in \partial \bH^2$,
$$d_p(\xi, \eta)= e^{-(\beta_\xi(p,q) +\beta_\eta(p, q))/2}$$
where $q$ is any point on the geodesic connecting $\xi$ and $\eta$. The diameter of $(\partial \bH^2, d_p)$ is equal to $1$.

For any $h\in H$,
we have $d_{p}(\xi, \eta)= d_{h(p)}(h(\xi), h(\eta))$.
For $ \xi\in \partial \bH^2$, and $r>0$, set
\begin{equation}\label{bp} B_p(\xi, r)=\{\eta\in \partial \bH^2: d_p(\eta, \xi)\le r\}\end{equation}
as was defined in the introduction.
Also, denote by $V(p, \xi, r)$
the set of all $\eta\in \partial \bH^2$ such that the distance between $p$ and the orthogonal projection of $\eta$ onto the geodesic $[p, \xi)$
is at least $r$. Note that 
\[
V(p, \xi, t )=B_p(\xi, \tfrac{e^{-t}}{\sqrt{1+e^{-2t}}}),
\]
see (\cite[Lemma 2.5]{Sc} and the discussion following that lemma). Therefore, 
$$V(p, \xi, r+1 )\subset
B_p(\xi, e^{-r})\subset V(p,\xi, r-1)\quad \text{for all $r\ge 1$.}$$

The following is a uniform version of Sullivan's shadow lemma \cite{Su}.
The proof of this proposition is similar to the proof of  
\cite[Thm.~3.2]{Sc}; since the dependence on the multiplicative constant is important to us,
we give a sketch of {}{the} proof while making the dependence of constants explicit.

 \begin{prop} \label{sq}  There exists a constant $c\asymp e^{\star d_Y}  $ such that for all $\xi\in \Lambda_Y$, $p\in \tilde C_0$, and $t>0$,
 \begin{multline*}    c^{-1}\cdot  \nu_p(F_{{\xi}_t})  \beta_Y e^{-\delta(Y) t +(1-\delta(Y))d(\xi_t, \Delta_Y(p))}
\le \nu_p(V(p, \xi, t)) \\  \le c \cdot \nu_p(F_{{\xi}_t}) e^{-\delta(Y) t +(1-\delta(Y))d(\xi_t, \Delta_Y(p))}\end{multline*} 
where\begin{itemize}
\item $\{\xi_t\}$ is the unit speed geodesic ray $[p,\xi)$ so that $d(p, \xi_t)=t$;
\item $F_{\xi_t}=\partial \bH^2$ if $\xi_t\in \Delta_Y\tilde C_0$, and
$F_{\xi_t}=F_{\xi_i}$ if  $\xi_t\in \Delta_Y\tilde C_i$ for $1\le i\le m$;
\item $\beta_Y:= \inf_{\eta\in \Lambda_Y, q\in \tilde C_0} \nu_q(B_q(\eta, e^{-d_Y})).$
\end{itemize}
 \end{prop}
\begin{proof} Let $p$, $\xi\in  \Lambda_Y$ and $\xi_t$ be as in the statement.
By the $\delta(Y)$-conformality of the PS density, we have
$$  \nu_p (V(p, \xi, t)) = e^{-\delta(Y) t} \nu_{\xi_t} (V(p, \xi, t)) .$$ 

Therefore it suffices to show
$$\nu_{\xi_t} (V(p, \xi, t))  \asymp  \nu_p(F_{{\xi}_t})\cdot  e^{(1-\delta(Y))d(\xi_t, \Delta_Y(p))}$$
while making the dependence of the implied constant explicit.

\medskip

\noindent{\bf Claim A.}  If $\xi_t\in \Delta_Y \tilde C_0$,
then \be\label{ca} 
e^{-\delta(Y) d_Y} \cdot \inf_{\eta\in \Lambda_Y} \nu_p(B(\eta, e^{-d_Y}))\ll  \nu_{\xi_t} (V(p, \xi, t)) \ll e^{\delta(Y) d_Y} |\nu_p| \ee
where the implied constants are absolute. 

First note that this implies the claim in the proposition if 
$\xi_t\in \Delta_Y \tilde C_0$. Indeed $d(\xi_t, \Delta_Y(p))\leq d_Y$ and $F_{\xi_t}=\partial\bH^2$ in this case. Moreover, by~\eqref{ca}, we have 
\[
e^{-\star d_Y}\beta_Y e^{-\delta(Y) t}\leq \nu_p (V(p, \xi, t)) = e^{-\delta(Y) t} \nu_{\xi_t} (V(p, \xi, t)) \leq  e^{\star d_Y} e^{-\delta(Y) t}
\]
where we also used $|\nu_p|=e^{\star d_Y}$ (recall that $p\in \tilde C_0$). Thus the claim in the proposition follows in this case. 

We now turn to the proof of Claim A. 
As  $\xi_t\in \Delta_Y \tilde C_0$, there exists $\gamma\in \Delta_Y$ such that $d(\xi_t, \gamma p)\le d_Y$.
Hence
\begin{align*}
e^{-\delta(Y) d_Y} \nu_{\xi_t} (V(p, \xi, t))  &\le \nu_{\gamma p} (V(p, \xi, t))=  
\nu_p (V(\gamma^{-1}p, \gamma^{-1}\xi, t))\\
&\le  e^{{\delta(Y)} d_Y} \nu_{\xi_t} (V(p, \xi, t)).
\end{align*}

 The upper bound  in \eqref{ca} follows from the first inequality, while
the lower bound follows from the second inequality; indeed 
\[
V(\gamma^{-1}p, \gamma^{-1}\xi, t)=V(\gamma^{-1}\xi_t, \gamma^{-1}\xi, 0)
\] 
and the latter contains $B_p(\gamma^{-1} \xi, e^{-d_Y})$, since $d(p, \gamma^{-1}\xi_t)\le d_Y$ and $d_Y\geq 1$.

\medskip

\noindent{\bf Claim B.} Let $\xi$ be a parabolic limit point in $\La_Y$. Assume that
for some $i\ge 1$, $\xi_t\in \tilde C_i$ for all large $t$.

We claim: 
\be\label{cb} 
\nu_{\xi_t} (V(p, \xi, t))\asymp \nu_p(F_\xi) \cdot e^{(1-\delta(Y))(d(\xi_t, \Delta_Y(p)) +d_Y)}
\ee
and
\be\label{cb3} 
\nu_{\xi_t} (\partial \bH^2-V(p, \xi, t))\asymp \nu_p(F_\xi) \cdot e^{(1-\delta(Y))(d(\xi_t, \Delta_Y(p)) +d_Y)}
\ee
where here and in what follows implied constants are of the form $e^{\pm\star d_Y}$ unless otherwise is stated explicitly.   

Let $s_i\ge 0$ be such that $\xi_{s_i}\in \partial \tilde C_i$. Then 
for all $t\ge s_i$, 
\[
|d(\xi_t, \Delta_Y(p)) -(t-s_i)|\le d_Y.
\]
Hence for \eqref{cb}, it suffices to show
\be\label{cb2} 
\nu_{\xi_t} (V(p, \xi, t))\asymp e^{(1-\delta(Y))(t-s_i)}\nu_p(F_\xi).
\ee
Note that if we set $\Delta_{Y,\xi}=\op{Stab}_{\Delta_Y}(\xi)$,
\[ 
\nu_{\xi_t} (V(p, \xi, t)) \asymp \sum_{\gamma \in \Delta_{Y, \xi}, \gamma F_\xi \cap V(p, \xt, t)\ne \emptyset} \nu_{\xi_t}(\gamma F_\xi)
.\]

Let $F^*_{\xi}$ denote the image of $F_\xi$ on the horocycle based at $\xi$ passing through $p$ via the inverse of the visual map. Since $p\in\tilde C_0$, there exists $\gamma\in\Delta_{Y, \xi}$ so that $\gamma F^*_{\xi}$ is contained in the closure of $\tilde C_0$. Hence, 
\[
\text{diam}F^*_{\xi}\le d_Y=\max\{1,\text{diam}(\tilde C_0)\}.
\]
We now apply \cite[Lemma 2.9]{Sc} with $K=F^*_{\xi}$ and let $K_3$ be as in loc.\ cit. By the definition of $K_3$ given in the proof of \cite[Lemma 2.9]{Sc}, we have $K_3\ll \text{diam}F^*_{\xi}$ where the implied constant is absolute. In view of \cite[Lemma 2.9]{Sc}, thus, if $\gamma\in\Delta_{Y, \xi}$ is so that $\gamma F_\xi \cap V(p, \xi, t)\ne \emptyset$, then $d(p, \gamma p)\ge 2t -kd_Y$, where $k$ is absolute. In consequence, 
$$ \nu_{\xi_t} (V(p, \xi, t)) \asymp \sum_{\gamma \in \Delta_{Y, \xi}, d(p, \gamma p)\ge 2t} \nu_{\xi_t}(\gamma F_\xi)$$
where the implied constant is absolute.

Now we use the fact that if $d(p, \gamma p)\ge 2t$, then for all $\eta\in F_\xi$,
 \[
 |\beta_\eta (\gamma^{-1}\xi_t, \xi_t) -d(p, \gamma p) +2t|\ll \text{diam}F^*_{\xi}\le d_Y
 \] 
 (cf.~proof of \cite[Lemma 2.9]{Sc}).
 Since 
 \[
 \nu_{\xi_t}(\gamma F_\xi)=\int_{\gamma F_\xi}d\nu_{\xi_t}=\int_{F_\xi}  e^{-\delta(Y) \beta_{\gamma\eta}(\xi_t, \gamma \xi_t))}d\nu_{\xi_t}(\eta),
 \]
 and $\nu_{\xi_t}(F_\xi)= e^{-\delta(Y) t}\nu_p(F_\xi)$,
we deduce, with multiplicative constant $\asymp e^{\delta(Y) d_Y}$,
 \begin{align*} &\sum_{\gamma \in \Delta_{Y, \xi}, d(p, \gamma p)\ge 2t} \nu_{\xi_t}(\gamma F_\xi) \asymp  \sum_{\gamma \in \Delta_{Y, \xi}, d(p, \gamma p)\ge 2t} e^{2\delta(Y) t -\delta(Y) d(p, \gamma p) } \nu_{\xi_t}(F_\xi) \\  & \asymp   \nu_p(F_\xi)  e^{\delta(Y) t }\sum_{\gamma \in \Delta_{Y, \xi}, d(p, \gamma p)\ge 2t} e^{-\delta(Y) d(p, \gamma p)}\\ &
\asymp  \nu_p(F_\xi)  e^{(1-\delta(Y)) t }
\end{align*}
using $a_n:=\#\{\gamma \in \Delta_{Y, \xi}: n< d(p, \gamma p)\le n+1\} \asymp e^{n/2}$ in the last estimate.
This proves \eqref{cb}.
 
The estimate \eqref{cb3} follows similarly now using 
 $$ \nu_{\xi_t} (\partial \bH^2-V(p, \xi, t)) \asymp \sum_{\gamma \in \Delta_{Y, \xi}, d(p, \gamma p)\le 2t} \nu_{\xi_t}(\gamma F)$$
 and $\sum_{n=0}^{[2t]} a_n e^{-\delta(Y) n} \asymp e^{(1-2\delta(Y) )t}$.
 
 Note that when $\xi$ is a parabolic limit point, \eqref{cb} holds with multiplicative constant $\asymp e^{\star d_Y}$ (see the proof of \cite[Prop.~3.4]{Sc}).
 
 As for the remaining case, i.e., $\xi$ is a radial limit point but 
 $\xi_t\in \Delta_Y \tilde C_i$ for some $i$, one can prove that 
  \eqref{cb} holds with multiplicative constant $\asymp e^{\star d_Y}$ (see the proof of \cite[Lemma 3.6]{Sc}).
  \end{proof}
  
\begin{prop}\label{thm:shadow-lemma} \label{sha}
 Fix $p=p_Y\in \tilde C_0$. There exists $R_Y\asymp e^{\star d_Y}$ such that for all $y\in Y_0$, we have 
\[
{R_Y^{-1}}{\beta_Y }  e^{(1-\delta(Y))d(C_Y, \pi(y))} |\nu_p| \leq
{\mu_y([-1,1])} \leq 
R_Y     e^{(1-\delta(Y)) d(C_Y, \pi(y))} |\nu_p|
\]
where $\pi$ denotes the base point projection $\Delta_Y\ba H =\T^1(S_Y) \to S_Y$.
\end{prop} 
\begin{proof} The following argument is a slight modification of the proof of \cite[Prop.~5.1]{SM}. 
Since the map $y\mapsto \mu_y[-1,1]$ is continuous on $Y_0$ and $\{[h]\in Y_0: h^-\text{ is a radial limit point of $\La_Y$}\}$
is dense in $Y_0$,
it suffices to prove the claim for $y=[h]$, assuming that $h^-$ is a radial limit point for $\Delta_Y$. 

{}{Recall that ${\mu_y([-1,1])}=e^{\delta(Y) t} {\mu_{ya_{-t}}([-e^{-t}, e^{-t}])}$ for all $t\in\mathbb R$.
Let $t\ge 0$ be the minimal number so that $\pi(ya_{-t})\in C_Y$; this exists as $h^-$ is a radial limit point. Then
\be\label{eq:diam-CY}
d(\pi(y), C_Y)\leq d(\pi(y), \pi(ya_{-t}))\leq  d_Y + d(\pi(y), C_Y).
\ee
}

Set $\xi_t=ha_{-t}(o)$.
Then 
 $$\mu_{ya_{-t}} [-e^{-t}, e^{-t}]\asymp \nu_{\xi_t}(V(\xi_t, h^+, {t}))$$ 
 (cf.~\cite[Lemma 4.4]{Sc}).

Since $ya_{-t}\in C_Y$, $F_{\xi_t}=\partial \bH^2$.
So $\nu_{\xi_t} (F_{\xi_t} )=|\nu_{\xi_t}|\asymp |\nu_p|$ up to a multiplicative constant $e^{\star d_Y}$.
Therefore, for some implied constant $\asymp e^{\star d_Y}$, we have
\begin{multline*}
\beta_Y   e^{-\delta(Y) t +(1-\delta(Y))d(\pi(y), \pi(ya_{-t}))} |\nu_p| \ll 
\nu_{\xi_t}(V(\xi_t, h^+, t)) \ll \\
e^{-\delta(Y) t +(1-\delta(Y))d(\pi(y), \pi(ya_{-t}))} |\nu_p|.
\end{multline*}
 
This estimate and~\eqref{eq:diam-CY}, therefore, imply that  
\[
\beta_Y    e^{(1-\delta(Y))d(\pi(y), C_Y)} |\nu_p| \ll \mu_y([-1,1]) \ll  e^{(1-\delta(Y))d(\pi(y), C_Y)} |\nu_p|
\]
with the  implied constant $\asymp e^{\star d_Y}$, proving the claim.
\end{proof}

 We use
the following result, essentially obtained by Schapira-Maucourant (\cite{Su}, \cite{SM}):
\begin{cor} \label{lower0}  
Fix $\rho>0$. Then for all $0<\e\leq \rho$, 
\[
{ R_Y^{-2}\cdot \beta_Y}\le \sup_{y\in Y_0} \frac{\mu_{y}([-\e,\e])}{\e^{\ds}\mu_{y}([-1,1])}\le 
\max\{1,\rho^2\}\cdot R_Y^2\cdot {\beta_Y^{-1}} <\infty,
\]
{}{where $R_Y$ is as in Proposition~\ref{sha}.}
\end{cor}
\begin{proof}  By \eqref{p1}, we have
$\mu_y([-\e,\e])=\e^{\delta(Y)} \mu_{ya_{-\log \e}}([-1,1]).$
Hence the case when $Y$ is convex cocompact follows from Proposition \ref{sha}.

Now suppose that $Y$ has a cusp. 
Let $y\in Y_0$.  Using the triangle inequality, we get that $d(\pi(ya_{-\log \e}), C_Y)-d(\pi(y), C_Y) \le |\log \e|$. 
Therefore, by Proposition \ref{sha}, we have 
\begin{align*}
\tfrac{\mu_{ya_{-\log \e}}([-1,1])}{\mu_y([-1,1])} &\le R_Y^2 \beta_Y^{-1}  \cdot e^{(1-\delta(Y)) (d(\pi(y a_{-\log \e}), C_Y)-d(\pi(y), C_Y) ) }\\
&\le \begin{cases} R_Y^2 \cdot\beta_Y^{-1}\cdot \e^{\delta(Y)-1}&\text{if $0<\e <1$}\\
 R_Y^2 \cdot\beta_Y^{-1}\cdot\e^{1-\delta(Y)}& \text{if $\e\geq 1$}\end{cases}.
\end{align*}

As a consequence, we have 
\[
\tfrac{\mu_y([-\e,\e])}{\e^{2\delta(Y)-1} \mu_y([-1,1])} \le \begin{cases} R_Y^2 \cdot\beta_Y^{-1}&\text{if $0<\e <1$}\\
 R_Y^2 \cdot\beta_Y^{-1}\cdot\rho^2& \text{if $\rho\geq 1$ and $1\leq \e\leq \rho$}\end{cases}.
\]
Recall from~\eqref{eq: def delta(Y)} that $\ds=\delta(Y)$ when $Y$ is cocompact and $\ds=2\delta(Y)-1$ otherwise. The above thus establishes the upper bound.   

%

By choosing $y\in Y_0$ such that $d(\pi(ya_{-\log \e}), C_Y)-d(\pi(y), C_Y) =|\log \e|$,
we get the lower bound.
\end{proof}

Theorem \ref{spp} follows from the following:
\begin{prop} \label{lower} 
We have
\begin{enumerate}
\item for any $0<\e\le 1/2$, $0<\mathsf s_Y(\e)<\infty$.
\item  $ \mathsf s_Y\asymp {\mathsf p}_Y  \ll  {e^{\star d_Y/\ds} } \beta_Y^{-1/\ds}.$
\end{enumerate}
\end{prop}

\begin{proof} 
Let $y\in Y_0$ and $h\in H$ be so that $y=[h]$. Fix $0<r\le 2$.
Recall
\[
\mu_y([-r,r])=\int_{-r}^r e^{-\delta(Y)\beta_{hu_s^+} (h(o), hu_s(o)) }d\nu_{h(o)} (hu_s^+).
\] 
Since $|\beta_{hu_r^+} (h(o), hu_r(o))|\le d(o, u_r(o))$, we have  $$e^{-\delta(Y)\beta_{hu_r^+} (h(o), hu_r(o)) }\asymp 1$$
with the implied constant independent of all $0<r\le 2$.

Since $d_o(u_r^+, e^+)= d_{h(o)} ((hu_r)^+, h^+)$ where $e$ is the identity (recall that $v_o^+=e^+$), we have
$$ \nu_{h(o)}(B_{h(o)}(h^+,   \tfrac{c^{-1} r}{\sqrt{1+2r^2}})) \ll
\mu_y([-r,r]) \ll  \nu_{h(o)}(B_{h(o)}(h^+, \tfrac{c r}{\sqrt{1+2r^2}}))$$
for some $c>1$ independent of $r$ and $h$.

This implies that 
\[
\mu_y([{-\e}/{c'} , {\e}/{c'}] ) \ll \nu_{h(o)}(B_{h(o)}(h^+, \e))\ll \mu_y([{-c'\e}, {c' \e}]) 
\]
as well as
\[
\frac{\mu_y([{-\e}/{c'} , {\e}/{c'}] )}{\e^{\ds}\mu_y([{-c'}/{2},{c'}/{2}])}\ll\frac{\nu_{h(o)}(B_{h(o)}(h^+, \e))}{\e^{\ds}\nu_{h(o)}(B_{h(o)}(h^+, 1/2))}\ll \frac{\mu_y([{-c'\e}, {c' \e}]) }{\e^{\ds}\mu_y([{-1}/({2c'}),{1}/({2c'})])}
\]
where $c'>1$ is independent of $0<\e<1/2$ and $h\in H$.

First note that by Corollary \ref{lower0}, we have   
\[
\mu_y([{-1}/({2c'}),{1}/({2c'})])\asymp_{c'}\mu_{y}[-1,1]\asymp_{c'}\mu_y([{-c'}/{2},{c'}/{2}])
.\]
Similarly, using Corollary \ref{lower0}, for any $0<\e\leq 1/2$, we have 
\[
\mu_y([{-\e}/{c'} , {\e}/{c'}])\asymp_{c'}\mu_{y}[-4\e,4\e]\asymp_{c'}\mu_y([-c'\e,c'\e]);
\]
the choice of the constant $4$ here is motivated by the definitions of $\mathsf p_Y$ and $\mathsf s_Y$ in~\eqref{sc} and~\eqref{eq:def-b-Y2}, respectively. 

Altogether we conclude that 
\[
\frac{\nu_{h(o)}(B_{h(o)}(h^+, \e))}{\e^{\ds}\nu_{h(o)}(B_{h(o)}(h^+, 1/2))}\asymp \frac{\mu_y([-4\e, 4\e]) }{(4\e)^{\ds}\mu_y([-1,1])}.
\]
Taking supremum over $0<\e\leq 1/2$ and $h\in H$ with $h^\pm\in\Lambda_Y$, we conclude that $\mathsf s_Y\asymp\mathsf p_Y$.  

The last claim follows from Corollary \ref{lower0}.
\end{proof}

\section{Linear algebra lemma}\label{sec:linear-alg}
The goal of this section is to prove the linear algebra lemma (Lemma \ref{lem:vect-contr}) and its slight variant (Lemma \ref{la}).

In this section, it is more convenient to identify $G$ as $\SO(\mathsf Q)^\circ$
 for the quadratic form $$\mathsf Q(x_1,x_2,x_3,x_4)=2x_1x_4-x_2^2-x_3^2.$$
   As $\mathsf Q$ has signature $(1,3)$, $\PSL_2(\c)\simeq \SO(\mathsf Q)^\circ$ as real Lie groups.
   We consider the standard representation of $G$ on the space $\bbr^4$ of row vectors and denote the Euclidean norm on $\bbr^4$ by $\|\cdot\|$.
  We have $$H=\op{Stab}_G(e_3)\simeq \SO(1,2)^\circ,$$
 $$A=\{a_t=\text{diag}(e^t, 1,1, e^{-t}):t\in \br\}<H\quad\text{and}$$
 \[
 U=\left\{u_r=\begin{pmatrix}1 & 0 & 0 & 0\\ r & 1& 0 & 0\\
 0 & 0 & 1 & 0 \\ \frac{r^2}{2} & r & 0 & 1\end{pmatrix}: r\in\br\right\}< H.
 \]
 Set $$V:=\bbr e_1\oplus \bbr e_2\oplus\bbr e_4.$$ Then the  restriction of the standard representation of $G$ to
  $H$ induces a representation of $H$ on $V$, which is isomorphic to
 the adjoint representation of $H$ on its Lie algebra $\mathfrak{sl}_2(\br)$; in particular, it is irreducible.

 Note that  for each $t>0$,
$\mathbb R e_2=\{v\in V: va_t=v\}$, $\mathbb R e_1$ is the subspace  of all vectors with eigenvalues $>1$, and 
$\mathbb R e_4$ is the subspace of all vectors with eigenvalues $<1$.

Let $p:V\to \mathbb Re_1\oplus \mathbb Re_2$ and $p^+: V\to \mathbb Re_1$ denote the natural projections. 
Writing $v=v_1e_1+v_2e_2+v_4e_4$, a direct computation yields that for any $r\in \br$,
\be\label {pe} p(vu_r)=(v_1+v_2r+\tfrac{v_4r^2}{2} )e_1+ (v_2+v_4r)e_2\text{ and }\ee
$$p^+(vu_r)=(v_1+v_2r+\tfrac{v_4r^2}{2})e_1.$$

For a unit vector $v\in V$ and $\e>0$, define
$$
D(v,\e)=\{r\in[-1,1]: \|p(vu_{r})\| \le \e\}; $$
$$D^+(v,\e)=\{r\in[-1,1]: \|p^+(vu_{r})\| \le \e\}.$$

\begin{lem}\label{prime}
For all $0<\e<1/2$ and a unit vector $v\in V$, we have
$$\ell (D(v,\e)) \ll \e\quad\text{and}\quad \ell (D^+(v,\e)) \ll  \e^{1/2}$$
where $\ell$ denotes the Lebesgue measure on $\br$.
\end{lem}
\begin{proof}
Since we are allowed to choose the implied constant in the statement, it suffices to prove the lemma for $0<\e\leq 0.01$.

Writing $v=v_1e_1+v_2e_2+v_4e_4$, we have
\[\ell (D(v,\e))\le \ell \{r\in[-1,1]: |v_1+v_2r+ \tfrac{v_4r^2}2|\leq  \e\text{ and }|v_2+v_4r|\leq  \e\}.
\] 

If  $|v_4|\geq 0.01$, then
$$\ell (D(v,\e)) \le  \ell \{r\in[-1,1]: |v_2+v_4r|\leq \e\}\le  200  \e.$$

If $|v_4|< 0.01$ but $0.1\leq |v_2|\leq 1$, then for $r\in[-1,1]$,
we have $|v_2+v_4r|\ge 0.09$, and hence for all $\e\leq 0.01$,
 $$\ell (D(v,\e)) \le  \ell \{r\in[-1,1]: |v_2+v_4r|\leq  \e\}=0.$$
Now consider the case when $|v_4|\leq 0.01$ and $|v_2|\leq 0.1$. Then, since $\|v\|=1$,
we get that $|v_1|\geq 0.7$. Hence for all
$r\in [-1,1]$, $|v_1+v_2r+v_4r^2/2|> 0.5$.
In consequence,  for  all $\e<1/2$,
$$\ell (D(v, \e))\le \ell \{r\in[-1,1]: |v_1+v_2r+v_4r^2/2|\leq  \e\}= 0,$$
proving the estimate on $D(v,\e)$. 
To estimate $D^+(v,\e)$,
observe that $p^+(vu_{r})=(v_1+v_2r+\tfrac{v_4r^2}{2})e_1 $ is a polynomial map of degree at most $2$. Moreover, 
since $\|v\|=1$, we have 
\[
\max\{|v_1|, |v_2|, |v_4|\}\gg1.
\] 
Therefore, $\sup_{{r}\in [-1,1]}\|p^+(vu_{r})\|\gg 1$. The claim about $D^+(v,\e)$ now follows using Lagrange's interpolation, see~\cite{BG} for a more general statement. 
\end{proof}

For the rest of this section, we fix a closed non-elementary $H$-orbit $Y$.

\begin{lem}\label{lem:D-D+}
There exists an absolute constant $\hat b_0>0$ for which the following holds:
for any $y\in Y_0$  and  $0<\e <1$, we have
\be\label{eq:muz-D}
\sup_{ v\in V,\|v\|=1}\mu_{y}{}(D(v,\e))\leq \hat b_0 {\mathsf p}_Y^{\ds} \e^{\ds} \mu_y([-1,1]),
\ee
and 
\be\label{eq:muz-D+}
\sup_{ v\in V,\|v\|=1} \mu_{y}{}(D^+(v,\e))\leq \hat b_0  {\mathsf p}_Y^{\ds} \e^{{\ds/}{2}} \mu_y([-1,1])
\ee
where  ${\mathsf p}_Y$ is given as in \eqref{sc}.
\end{lem}

\begin{proof}
By \eqref{pe}, each set $D(v,\e)$ and $D^+(v,\e)$ consists of at most $2$ intervals. By Lemma \ref{prime},
$D(v,\e)$ (resp.\ $D^+(v,\e)$) may be covered by $\ll 1$ many intervals of length $\e$ (resp.\ $\e^{1/2}$).  Therefore \eqref{eq:muz-D}  (resp.~\eqref{eq:muz-D+}) follows from the definition of ${\mathsf p}_Y$. \end{proof}

 We use Lemma~\ref{lem:D-D+} to prove the following lemma which will be crucial in the sequel.

 \begin{lem}[Linear algebra lemma] \label{lem:vect-contr} 
For any $\frac{\ds}{3}\le s < {\ds}$, $1\le \rho \le 2$, and $t>0$, we have
 \be \label{gozero}
 \sup_{y\in Y_0, v\in V, \|v\|=1} \frac{1}{ \mu_y([-\rho,\rho])}\int_{-\rho}^{\rho} \frac{1}{\| vu_r a_t\|^s} d\mu_{y}{} (r) \le
b_0 \frac{{\mathsf p}_Y^{\ds} e^{-(\ds- s)t/4} }{(\ds-s)}
 \ee
where $b_0\ge 2$ is an absolute constant. \end{lem}

\begin{proof}
We first claim that it suffices to prove the claim for $\rho=1$. Indeed, let $t_\rho=t-\log\rho$ and let $y_\rho=ya_{-\log\rho}$, and for every $v\in V$, let $v_\rho=va_{-\log\rho}$. 
Recall also that $\mu_y[-r,r]=\rho^{\delta(Y)} \mu_{ya_{-\log \rho}}{} [-r/\rho,r/\rho]$ and that $Y_0$ is $A$-invariant.
Thus, 
\begin{align*}
\frac{1}{ \mu_y([-\rho,\rho])}\int_{-\rho}^{\rho} \frac{1}{\| vu_r a_t\|^s} d\mu_{y}(r)&=\frac{1}{ \mu_y([-\rho,\rho])}\int_{-\rho}^{\rho}\frac{1}{\| va_{-\log\rho}u_{\rho^{-1}r} a_{t_\rho}\|^s} d\mu_{y}(r)\\
&=\rho^{\delta(Y)}\|v_\rho\|^{-s}\frac{1}{ \mu_{y_\rho}([-1,1])}\int_{-1}^{1}\frac{1}{\| v_\rho' u_{r} a_{t_\rho}\|^s} d\mu_{y_\rho}(r)
\end{align*}
where $v'_\rho=v_\rho/\|v_\rho\|$.

Since $\|v_\rho\|^{-s}\asymp 1$ (with absolute implied constants for $1\leq \rho\leq 2$) and $Y_0$ is $A$-invariant, it thus suffices to prove the lemma for $\rho=1$.  

Fix $0<s<\ds$ and $t>0$.
We observe that for all $r\in \br$,
\be\label{eq:ncont-exp}
\text{$\|vu_ra_t\|\geq \|p(vu_r)\|\;\;$ and $\;\;\| v u_{r} a_t\| \ge  e^{t} \|p^+(vu_{r})\|$.}
\ee

For simplicity, set $\beta_y:=\frac{1}{\mu_{y}{}([-1,1])}$.  The inequality \eqref{eq:muz-D}  and the
 first estimate in~\eqref{eq:ncont-exp} imply that
for any $0<\e\leq 1$ and any unit vector $v\in V$, we have
 \begin{align*}
\beta_y \int_{{r}\in D(v,\e )- D(v, \e /2)} {\| vu_{r}a_t\|^{-s}} d\mu_{y}{}({r}) &\le {\hat b_0 {\mathsf p}_Y}^{\ds} \e^{\ds} \cdot (\e/2)^{-s} \\
&\le  2\hat b_0 {\mathsf p}_Y^{\ds} \e^{\ds-s}.
\end{align*}
We write $D(v,\e)=\bigcup_{k=0}^\infty D(v, \e/2^k) -D(v, \e/2^{k+1})$.
Now applying the above estimate for each $\e/2^k$ and summing up the geometric series, we get that
  for any $0<\e<1$, 
 \be\label{eq:geom-sum-est}
\beta_y \int_{{r}\in D(v, \e)}  {\| vu_{r}a_t\|^{-s}} d\mu_{y}{}({r}) 
\le \frac{2\hat b_0 {\mathsf p}_Y^{\ds} \e^{\ds-s}}{1 -2^{s-\ds}} .
 \ee

Moreover, using~\eqref{eq:muz-D+} and the first estimate in~\eqref{eq:ncont-exp} again, for any $\kappa>0$, we have 
 \be\label{eq:D+-not-D}
\beta_y  \int_{{r}\in D^+(v,\kappa)- D(v, \e)}  {\| vu_{r}a_t\|^{-s}} d\mu_{y}{}({r})\leq 2\hat b_0 {\mathsf p}_Y^{\ds}\kappa^{\ds/2}\e^{-s}.
 \ee
 Finally, the definition of $D^+(v, \kappa)$ and the second estimate in~\eqref{eq:ncont-exp} imply  
  \be\label{eq:easy-est}
\beta_y \int_{{r}\in [-1,1]-D^+(v,\kappa)}  {\| vu_{r}a_t\|^{-s}} d\mu_{y}{}({r})\leq \kappa^{-s} e^{-st}.
 \ee
Combining~\eqref{eq:geom-sum-est},~\eqref{eq:D+-not-D}, and~\eqref{eq:easy-est} and using 
the inequality $\tfrac{1}{1-2^{-(\ds-s)}} \le \tfrac{2}{\ds-s}$,
we deduce that for any $0<\e, \kappa <1$,
$$
\beta_y \int_{-1}^1  {\| vu_{r}a_t\|^{-s}} d\mu_{y}{}({r})\le \frac{2\hat b_0\mathsf p_Y^{\ds} }{\ds-s} \left(\e^{\ds-s} +\kappa^{\ds/2} \e^{-s}
+\kappa^{-s} e^{-st}\right).
$$

Let $\epsilon=e^{-t/4}$ and $\kappa={\epsilon}^{2}$.
As $\ds/3 \le s<\ds$, we have $e^{-s/2}\le e^{(s-\ds)/4}$.
This yields: \[
\beta_y  \int_{-1}^1  {\| vu_{r}a_t\|^{-s}} d\mu_{y}{}({r})\leq \frac{6\hat b_0\mathsf p_Y^{\ds} }{\ds-s}  \cdot e^{-(\ds-s)t/4},
\]
as we claimed.
\end{proof}

We will extend the upper bound in Lemma \ref{lem:vect-contr}  to all unit vectors $v\in e_1G$,
based on the fact that the vectors in $e_1G$  are projectively away from the $H$-invariant point corresponding to $\br e_3$.

\begin{lem}\label{BONE} 
There exists an absolute constant $b_1>1$ such that for any vector $v\in e_1G\subset \br^4$,
$$ \|v\| \le  b_1 \|v_1\|$$
where $v_1$ is the projection of $v\in \br^4$ to $V=\bbr e_1\oplus \bbr e_2\oplus\bbr e_4$.
\end{lem}
\begin{proof}
Since $\mathsf Q(e_1)=0$ and $G=\SO(\mathsf Q)^\circ$, we have $\mathsf Q(e_1g)=0$ for every $g\in G$. 
 Since
$\mathsf Q(e_3)=-1$, the set $\{\|v\|^{-1} v: v\in e_1G\}$ is a compact subset  of the unit sphere in $\br^4$ not containing $\pm e_3$.
Therefore there exists an absolute constant $0<\eta<1$ such that
if we write $v=v_1 +re_3\in e_1G$, then $|r|\le \eta \|v\|$.
Therefore $\|v_1\|^2=\|v\|^2 -r^2 \ge (1-\eta^2)\|v\|^2$. Hence it suffices to set $b_1=(1-\eta^2)^{-1/2}$.
\end{proof}

\begin{lem} [Linear algebra lemma II] \label{la}
For any $\frac{\ds}{3}\le s < {\ds}$, $1\le \rho \le 2$, and $t>0$, we have
 \begin{equation*} \sup_{y\in Y_0,v\in e_1G, \|v\|=1} \frac{1}{ \mu_y([-\rho,\rho])}\int_{-\rho}^{\rho} \frac{1}{\| vu_r a_t\|^s} d\mu_{y}{} (r) \le
b_0 b_1\frac{{\mathsf p}_Y^{\ds} e^{-(\ds- s)t/4} }{(\ds-s)}
 \end{equation*}
where $b_0\ge 2$ and $b_1>1$ are absolute constants as in Lemmas \ref{lem:vect-contr} and \ref{BONE} respectively. \end{lem}

\begin{proof}
Let $v\in e_1G$ be a unit vector, and write $v=v_0+v_1$ where $v_0\in\bbr e_3$ and $v_1\in V$.
Since $e_3$ is $H$-invariant, 
we have $vh=v_0+v_1h\in \bbr e_3\oplus V$ for all $h\in H$. Therefore,
\begin{align*}
&\frac{1}{\mu_y([-\rho,\rho])}\int_{-\rho}^{\rho} \frac{1}{\| vu_r a_t\|^s} d\mu_y(r)\leq 
\frac{1}{\mu_y([-\rho,\rho])}\int_{-\rho}^{\rho} \frac{1}{\| v_1u_r a_t\|^s} d\mu_y(r)\\
\\ & \le \frac{b_0{\mathsf p}_Y^{\ds} e^{-(\ds-s)t/4} }{(\ds-s)} \|v_1\|^{-s}\quad \text{ by Lemma~\ref{lem:vect-contr}} \\
&\le   \frac{ b_0b_1{\mathsf p}_Y^{\ds} e^{-(\ds- s)t/4} }{(\ds-s)} \|v\|^{-s}\quad\text{by Lemma \ref{BONE}} .
\end{align*}

\end{proof}

\section{Height function $\omega$}\label{sec:height-func}
In this section we define the height function $\omega:X_0\to (0, \infty)$ and show that 
$\omega(x)$ is comparable to the reciprocal of the injectivity radius at $x$. 

 For this purpose, we continue to realize $G$ as $\SO(\mathsf Q)^\circ$ acting on $\br^4$ by the standard representation,
 as in Section \ref{sec:linear-alg}.
Observe that $\mathsf Q(e_1)=0$ and the stabilizer of $e_1$ in $G$ is equal to $M_0N$.

Fixing  a set of $\Gamma$-representatives $\xi_1, \cdots, \xi_\ell$  in  $\La_{bp}$, choose  elements  $g_i\in G$ so that  $g_i^-=\xi_i$ and
$\|e_1g_i^{-1}\|=1$; this is possible since
$\{g\in G: g^-=\xi_i\}$ is a conjugate of $AM_0N$.

 Set \be\label{vi} v_i:= e_1g_i^{-1}\in e_1G.\ee
Note that
\[
{\rm Stab}_G(\xi_i)=g_i  A \Sone Ng_i^{-1}\text{ and } \op{Stab}_G(v_i)= g_i\Sone Ng_i^{-1} .
\] 

By Witt's theorem, we have that for each $i$,
$$\{v\in \br^4-\{0\}: \mathsf Q(v)=0\}=v_iG\simeq  g_i\Sone Ng_i^{-1}\ba G.$$

\begin{lem} \label{discrete} For each $1\le i\le \ell$, the orbit $v_i\Gamma$ is a closed (and hence discrete)
subset of $\br^4$.
\end{lem}
\begin{proof}  
The condition $\xi_i\in \La_{bp}$ implies that $\Gamma\ba \Gamma g_i M_0N$ is a closed subset of $X$. Equivalently,
$\Gamma g_i M_0N$ as well as $ \Gamma g_i M_0N g_i^{-1}$ is closed in $G$. Therefore, its inverse
$ g_i\Sone Ng_i^{-1}\Gamma$ is a closed subset of $G$.  In consequence, $v_i\Gamma\subset\bbr^4$ is a closed subset
of $v_iG=\{v\in \br^4-\{ 0\}: \mathsf Q(v)=0\}$.

It remains to show that $v_i\Gamma$ does not accumulate on $0$. Suppose on the contrary that
there exists an infinite sequence $v_i\gamma_\ell $ converging to $ 0$ for some $\gamma_\ell \in \Gamma$.
Using the Iwasawa decomposition $G=g_i NA K_0$, we may write 
$\gamma_\ell =g_i n_\ell a_{t_\ell} k_\ell$ with $n_\ell\in N, t_\ell\in \br$
and $k_\ell \in K_0$.
Since $$v_i\gamma_\ell = e^{t_\ell} (e_1 k_\ell ),$$  the assumption that $v_i\gamma_\ell \to 0$ implies that $t_\ell \to -\infty$.

On the other hand, as $\xi_i\in \Lambda_{bp}$,
 $\op{Stab}_\Gamma (\xi_i)=\Gamma\cap g_i AM_0Ng_i^{-1}$ contains a parabolic element, say, $\gamma'\neq e$.  Note that $n_0:=g_i^{-1}
 \gamma' g_i $ is then  an element of $N$ and hence a unipotent element, as any parabolic element of $AM_0N$
belongs to $N$ in the group $G\simeq \PSL_2(\c)$.
Now observe that, as $N$ is abelian,
 $$\gamma_{\ell}^{-1}\gamma' \gamma_\ell=k_\ell ^{-1}a_{-t_\ell} ( n_\ell^{-1} g_i^{-1} \gamma' g_i n_\ell  ) a_{t_\ell}k_\ell = k_\ell ^{-1}
 (a_{-t_\ell} n_0 a_{t_\ell} ) k_\ell .
 $$

Since $t_\ell \to -\infty$,
 the sequence $a_{-t_{\ell}} n_0 a_{t_\ell}$ converges to $e$. Since $\{k_\ell^{-1}\}$ is a bounded sequence,
it follows that, up to passing to a subsequence, $\gamma_{\ell}^{-1}\gamma' \gamma_\ell$ is an infinite sequence converging to $e$, contradicting the discreteness of $\Gamma$.
\end{proof}

\begin{Def}  [Height function] \label{defh} \rm Define the height function $\omega:X_0\to [2, \infty)$ by
$$\omega(x):=\max_{1\le i\le \ell}\omega_i(x)$$
where 
\[
\omega_i(x)=\max_{\gamma\in \Gamma} \Bigl\{2, {\|v_i\gamma g\|^{-1}}\Bigr\}\quad\text{for any $g\in G$ with $x=[g]$};
\]
this is well-defined by Lemma \ref{discrete}.

If $\Gamma$ has no parabolic elements, we define $\omega(x)=2$ for all $x\in X_0$.
\end{Def}

By the definition of $\e_X$, $X_0$ is contained in the union of $X_{\e_X}$ and $\cup_{j=1}^{\ell}  \mathfrak h_j$ where
$\mathfrak h_j$ is a horoball based at $\xi_j$. 
\[
\text{Fix $T_j>0$ so that $\mathfrak h_j =[g_j] N{ A}_{(-\infty, -T_j]} K_0$.}
\]
Set $\tilde{\mathfrak h}_j:=g_j N{A}_{(-\infty, -T_j]} K_0$.

The following is an immediate consequence of the thick-thin decomposition of $M$:
\begin{lem}\label{WWW}
If $\tilde{\mathfrak h}_j\cap \gamma \tilde{\mathfrak h}_i\ne \emptyset$ for some $1\le i, j\le \ell$ and $\gamma\in \Gamma$,
then $i=j$, $\gamma\in \op{Stab}_G(\xi_i)=\op{Stab}\tilde{\mathfrak h}_i$, and hence
$\tilde{\mathfrak h}_j= \gamma \tilde{\mathfrak h}_i$.
\end{lem}

\begin{lemma}\label{ETA}
For all $1\le i, j \le \ell$
 and $\gamma\in \Gamma$ such that $\tilde{\mathfrak h}_j\neq \gamma \tilde{\mathfrak h}_i$, 
 \be\label{eq:reduction-theory}
\inf_{q\in \tilde{\frak h}_i} \| v_j \gamma h\| \ge  \rdt 
\ee
where $\rdt:=\min_{1\le m\le \ell} e^{-T_m}.$
\end{lemma}

\begin{proof} 
Let $q\in \tilde{\frak h}_i$ and $\gamma\in \Gamma$.
Using $G=g_jNAK_0$, write
 $\gamma q=g_j ua_{s}k\in g_jNAK_0$. Then $\|v_j\gamma q\|= e^s$.
 Hence if $\|v_j\gamma q\|<\eta_0$, then  $s\le  -T_j$.
 So  $\gamma q\in \tilde{\mathfrak h}_j$.  Therefore $\tilde{\mathfrak h}_j\cap \gamma\tilde{\mathfrak h}_i\ne \emptyset$.
By Lemma \ref{WWW},  $\tilde{\mathfrak h}_j=\gamma \tilde{\mathfrak h}_i$. 
\end{proof}

\begin{prop}\label{lem:cusp-function}\label{cu}\label{alpha} 
There is an absolute constant $\alpha\ge 2$ such that for all $x\in X_0$,
\be\label{eq:alpha}
\tfrac{1}{2\alpha} \cdot \inj (x)\le  \omega(x)^{- 1} \le  \tfrac{\alpha}{2} \cdot \inj (x).
\ee
\end{prop}

\begin{proof} 
Fixing $1\le j\le \ell$, it suffices to show the claim for  all $x\in X_0\cap \mathfrak h_j$.

Let $g\in g_i u a_{-t} k \in \tilde{ \mathfrak h}_i$ be so that $x=[g]$,
where $ua_{-t}k \in  N{ A}_{(-\infty, -T_j]} K_0$.
 
Note that
\begin{equation*}
\omega_i(x)^{-1} \le  \|v_i g\| = \|e_1g_i^{-1} (g_i u a_{-t}k)\|=\|e_1 u a_{-t}k\|=e^{-t}.
\end{equation*}
In view of the definition of $\omega$ and $\omega_i$, this together with Lemma \ref{ETA} implies that
$$\omega(x)=\omega_i(x)=e^{t}.$$ 
Since $\op{inj}(x) \asymp e^{-t}$, this finishes proof.  
\end{proof}

 \section{Markov operators} \label{AVE}
In this section we define a Markov operator ${\mathsf A}_t$  and prove
Proposition \ref{prop:iteration} which relates  the average $m_Y(F)$ of a locally bounded, log-continuous, Borel function $F$ on $Y_0$ with
a super-harmonic type inequality for ${\mathsf A}_t F$. This proposition  will serve as a main tool in our approach to prove Theorem \ref{main}.

Fix a closed non-elementary $H$-orbit $Y$ in $X$.

\subsection*{Bowen-Margulis-Sullivan measure $m_Y$} We denote by $m_Y$ the
 Bowen-Margulis-Sullivan probability measure on $\Delta_Y\ba H=\T^1(S_Y)$, which is the unique 
 probability measure of maximal entropy (that is $\delta(Y))$ for the geodesic flow.
We will also use the same notation $m_Y$ to denote the push-forward of the measure to $Y$ via the map 
$\op{Stab}_H(y_0)\ba H\to Y$ given by $[h]\to y_0h$. Considered as a measure on $Y$, $m_Y$ is well-defined, independent
of the choice of $y_0\in Y$.

Recall the definition of $Y_0$ in \eqref{yay}; note that $Y_0=\op{supp}m_Y$.
 In the following, all of our Borel functions are assumed to be defined everywhere in their domains. By a locally bounded function, we mean a function which is bounded on every compact subset.
\begin{Def} [Markov Operator]\label{Markov} Let $t\in \br$ and $\rho>0$. For a locally bounded Borel function $\psi:Y_0\to \br$, we define
\be\label{eq:def-A-tau}
({\mathsf A}_{t, \rho} \psi) (y):=\frac{1}{\mu_y([-\rho,\rho])}\int_{-\rho}^{\rho} \psi (y u_r a_t) d\mu_y(r).
\ee
We set ${\mathsf A}_t:={\mathsf A}_{t, 1}$. 
\end{Def}

Note that ${\mathsf A}_{t, \rho} \psi$ is a locally bounded Borel function on $Y_0$.
Although $\lim_{n\to \infty}\mathsf A_{nt} (\psi)= m_Y(\psi)$ for any $\psi\in C_c(Y_0)$ and any $t>0$ \cite{OS},
the {\it Margulis function} $F$ we will be constructing is not a continuous function on $Y_0$, and hence we cannot use
such an equidistribution statement to control $m_Y(F)$. We will use the following lemma instead:

\begin{lem}\label{lem:F-integ}
Let $F:Y_0\to[2,\infty)$ be a locally bounded Borel function. Assume that there exist  some $t>0$ and $D>0$ such  that
\be\label{finite} 
\limsup_{n\to \infty} {\mathsf A}_{nt} F(y)\leq D \quad\text{ for all $y\in Y_0$}. 
\ee
Then $$m_Y(F)\leq 8D.$$ 
\end{lem}

\begin{proof} For every $k\ge 2$, let  $F_k:Y_0\to [2, \infty)$ be given by 
$$F_k(y):=\min\{F(y), k\}.$$
As  $F_k$ is bounded, it belongs to $L^1(Y_0, m_Y)$. 
Since the action of $A$ is mixing for $m_Y$ by the work of Babillot~\cite{Babillot}, we have $m_Y$ is $a_t$-ergodic for each $t\ne 0$. 
Hence, by the Birkhoff ergodic theorem, for $m_Y$-a.e.\ $y\in Y_0$,  we have  
$$\lim_{N\to \infty} \frac{1}{N}\sum_{n=1}^N F_k(ya_{nt}) =\int F_k\,dm_Y.$$ 
Therefore, using Egorov's theorem, for every $\e>0$, there exist $N_\e>1$ and
a measurable subset $Y_\e'\subset Y_0$ with $m_Y(Y_\e')>1-\e^2$ such that for every $y\in Y_\e'$ and all $N>N_\e$, we have  
\[
\frac{1}{N}\sum_{n=1}^N F_k(ya_{nt})>\frac{1}{2}\int F_k\,dm_Y.
\] 
Now by the maximal ergodic theorem~\cite[App.~A.1]{Lin-QUE}, 
if $\e$ is small enough, there exists a measurable subset $Y_\e\subset Y'_\e$ with $m(Y_\e)>1-\e$ so that
for all $y\in Y_\e$, we have 
\[
\mu_y\{r\in[-1,1]:yu_r\in Y_\e'\}>\tfrac12 \mu_y([-1,1]).
\]
Altogether, if $y\in Y_\e$ and $N>N_\e$, we have 
$$\tfrac 1 N \sum_{n=1}^N {\mathsf A}_{nt} F_k(y) =\tfrac{1}{\mu_y([-1,1])}\int_{-1}^{1}\tfrac{1}{N}\sum_{n=1}^N F_k(yu_ra_{nt}) d\mu_y(r)>{\tfrac{1}{4}}\int F_k\,dm_Y.$$

Fix $y\in Y_\e$. By the hypothesis \eqref{finite}, there exists  $n_0=n_0(y)$ such that for all $n\ge n_0$, we have
\begin{align*}
{\mathsf A}_{nt} F_k(y) \leq {\mathsf A}_{nt} F (y) \le 2D.
\end{align*}
Therefore, we deduce  that for all sufficiently large $N\gg 1$,
\begin{align*} 
\tfrac{1}{4}\int F_k\,dm_Y \le \tfrac{1}{N} \left( \sum_{n=1}^{n_0} {\mathsf A}_{nt}  F_k(y)  +\sum_{n=n_0+1}^N {\mathsf A}_{nt} F_k (y) \right) \le  \tfrac{k n_0}N +\tfrac{ 2D(N-n_0)}{N}.
\end{align*}
By sending $N\to \infty$, we get that for all $k>2$,
\[
\int F_k\,dm_Y\le 8 D.
\]

Since $\{F_k:k=3,4,..\}$ is an increasing sequence
of positive functions converging to $F$ point-wise, the monotone convergence theorem
implies 
\[
\int F\, dm_Y=\lim_{k\to\infty} \int  F_k\,dm_Y\le 8 D
\]
as we claimed. 
\end{proof}

We remark that  in \cite{EMM-UppB}, the Markov operator $\mathsf{A}_t$ was defined using the integral over the translates  $\SO(2) a_t$, whereas we use the integral over the translates {}{$U_{[-\rho,\rho]}a_t$} of a horocyclic piece. 
The proof of the following proposition, which is an analogue of~\cite[\S5.3]{EMM-UppB}, is the main reason for our digression from their definition, as the handling of the PS-measure on $U$ is more manageable than
that of the PS-measure on $\SO(2)$ in performing change of variables.

\begin{prop}\label{prop:iteration}
Let  $F:Y_0\to[2,\infty)$ be a locally bounded  Borel function satisfying the following properties:
\begin{itemize}
\item[(a)] There exists $\sigma\ge 2$ such that for all $h\in B_H(2)$ and $y\in Y_0$, 
\[ \sigma^{-1}F(y)\leq F(yh)\leq \sigma F(y). \]

\item[(b)] There exist $t\geq 2$ and  $D_0>0$ such that for all $y\in Y_0$ and $1\le \rho\le 2$,
$$
 {\mathsf A}_{t,\rho}F(y)\leq\frac{1}{8 \sigma \mathsf p_Y^{\ds}} \cdot F(y)+D_0 
$$
where $\mathsf p_Y$ is as in~\eqref{sc}.
\end{itemize} 
Then  $$m_Y(F)\le 64 D_0\mathsf p_Y^{\ds}.$$
\end{prop}

In view of Lemma \ref{lem:F-integ}, Proposition \ref{prop:iteration} is an immediate consequence of the following:
\begin{prop}\label{prop:iteration2}
Let $F$ be as in  Proposition \ref{prop:iteration}. Then  for all $y\in Y_0$ and $n\geq 1$, we have 
\be\label{eq:it-prop}
{\mathsf A}_{nt}F(y)\leq \frac{1}{2^{n}} F(y)+8D_0\mathsf p_Y^{\ds}.
\ee
\end{prop}  

\begin{proof}
The main step of the proof is the following estimate. 

\noindent{\bf Claim:}
For any $1\leq \rho\leq \frac32$, $y\in Y_0$ and $n\in \mathbb N$, we have 
\be\label{eq:main-iterations}
{\mathsf A}_{(n+1)t,\rho}F(y)\leq \tfrac 12 {\mathsf A}_{nt, \rho+e^{-nt}}F(y)+\hat D
\ee
where $\hat D:=4  D_0  \mathsf p_Y^{\ds}$; recall that $e^{-nt}\leq 1/2$.

Let us first assume this claim and prove the proposition.
We observe 
\begin{itemize}
\item $\sum_{j\ge 1} e^{-jt}\leq 1/2$ (as $t\ge 2$),
\item $({8 \sigma \mathsf p_Y^{\ds}} )^{-1} \le 1/2$, and
\item $D_0\le \hat D$.
\end{itemize}
Using the assumption (b) of Proposition \ref{prop:iteration} with 
$\rho_n=1+\sum_{j=1}^{n-1} e^{-jt}$ ($n\geq 2$),  
we deduce that for any $n\geq 2$, 
\begin{align}\label{eq:Antau-kappa}
{\mathsf A}_{nt} F(y)& \leq \tfrac{1}{2^{n-1}} {\mathsf A}_{t, \rho_n}F(y)+\hat D(1+ \tfrac 12
+\cdots+ \tfrac 1{2^{n-2}})\notag \\ 
&\leq \tfrac 1{2^{n-1}} \left((8 \sigma \mathsf p_Y^{\ds})^{-1} F(y)+D_0\right)+\hat D(1+\tfrac 1{2} +\cdots+\tfrac 1{2^{n-2}})\notag\\
&\leq \tfrac 1{2^{n}} F(y)+2\hat D
\end{align}
which establishes the proposition.

We now prove the claim \eqref{eq:main-iterations}.
For $y\in Y_0$ and  $\rho>0$, set $$b_y(\rho):=\mu_{y}([-\rho,\rho])\text{ and } b_y=b_y(1).$$

 To ease the notation, we prove \eqref{eq:main-iterations} with $\rho=1$; the proof in general is similar.
By assumption (a) and (b) of Proposition \ref{prop:iteration}, we have 
\be\label{eq:cont-av-upp-av}
{\mathsf A}_t F(y)\leq c_0F(y)+D_0
\leq \Bigl(\frac{c_0\sigma }{b_y}\int_{-1}^1 F(yu_{r}) d\mu_{y}(r) \Bigr)+D_0
\ee
where $c_0=({8 \sigma \mathsf p_Y^{\ds}})^{-1}$.

Set $\rho_n:=e^{-nt}$.
Let $\{[r_j-\rho_n, r_j+\rho_n]: j\in J\}$ be a covering of 
\[
[-1,1]\cap {\rm supp}(\mu_{y})
\] 
with $r_j\in [-1,1]\cap {\rm supp}(\mu_{y})$ and with multiplicity bounded by $2$. 
For each $j\in J$, let $z_j:=yu_{r_j}$. Then
\be\label{eq:almost-disj}
\sum_jb_{z_j}(\rho_n)=\sum_j\mu_y([r_j-\rho_n, r_j+\rho_n])\leq 2b_y(2).
\ee

Moreover, we get 
\begin{align}
\notag {\mathsf A}_{(n+1)t} F(y)&=\frac{1}{b_y}\int_{-1}^1 F(yu_{r}a_{(n+1)t})d\mu_{y}{}(r)\\
\notag&\leq\frac{1}{b_y}\sum_j\int_{-\rho_n}^{\rho_n}F(z_ju_{r}a_{(n+1)t})d\mu_{{z_j}}{}(r)\\
\label{eq:local-comp-omega-1}&=\frac{1}{b_y}\sum_j\int_{-\rho_n}^{\rho_n}F(z_ja_{nt}u_{re^{nt}}a_t)d\mu_{z_j}{}(r).
\end{align}
We now make the change of variables $s=re^{nt}$. In view of~\eqref{eq:local-comp-omega-1},
we have  
\[
{\mathsf A}_{(n+1)t} F(y)\leq\frac{1}{b_y}\sum_j\frac{b_{z_j}(\rho_n)}{b_{z_ja_{nt}}}\int_{-1}^1F(z_ja_{nt}u_{s}a_t)d\mu_{z_ja_{nt}}{}(s).
\]
Applying~\eqref{eq:cont-av-upp-av} with the base point $z_ja_{nt}$,  we get from the above that
\begin{multline}\label{eq:local-comp-omega-2}
{\mathsf A}_{(n+1)t} F(y)\leq\frac{1}{b_y}\sum_j\frac{b_{z_j}(\rho_n)c_0\sigma}{b_{z_ja_{nt}}}\int_{-1}^1 F(z_ja_{nt}u_{s})d\mu_{z_ja_{nt}}{}(s)+\\ \frac{1}{b_y}\sum_jb_{z_j}(\rho_n)D_0.
\end{multline}
By~\eqref{eq:almost-disj}, we have $\frac{1}{b_y}\sum_jb_{z_j}(\rho_n)D_0\leq \hat D.$

Therefore, reversing the change of variable, i.e., now letting $r=e^{-nt}s$, we get from~\eqref{eq:local-comp-omega-2} 
the following:
\begin{align}
\notag {\mathsf A}_{(n+1)t} F(y)&\leq \frac{1}{b_y}\sum_j c_0\sigma\int_{-\rho_n}^{\rho_n} F(z_ju_{r}a_{nt}) d\mu_{z_j}{}(r)+\hat D\\
\notag &\leq \frac{2c_0\sigma}{b_y}\int_{-(1+\rho_n)}^{1+\rho_n} F(yu_{r}a_{nt}) d\mu_{y}{}(r)+\hat D\\
\notag&=\frac{2c_0\sigma b_y(1+\rho_n)}{b_y } {\mathsf A}_{nt, 1+\rho_n}F(y) +\hat D .
\end{align} 
Since $$ \sup_{y\in Y_0} \frac{2c_0\sigma b_y(2)}{b_y}= ({4 \mathsf p_Y^{\ds}})^{-1} \sup_{y\in Y_0} \frac{ b_y(2)}{b_y} \le \frac 12, $$
we get  $${\mathsf A}_{(n+1)t} F(y) \le \frac{1}{2} {\mathsf A}_{nt, 1+\rho_n}F(y) +\hat D .$$

The proof is complete.
\end{proof}

\section{Return lemma and number of nearby sheets}\label{RET}
We fix closed non-elementary $H$-orbits $Y$ and $Z$ in $X$.
Since $Z$ is closed, a fixed ball around $y\in Y_0$ intersects only finitely many sheets of $Z$ (Fig. 2).
The aim of this section is to show that the number of sheets of $Z$ in $B(y, \inj (y))$ is controlled by
the tight area of $S_Z$ with a multiplicative constant depending on $\mathsf p_Y$ and $\ds$.

The main ingredient is a return lemma which says that for any $y\in Y_0$,
 there exists some point in $\{yu_r\in Y_0: r\in [-1,1]\}$ whose minimum return time to 
 a fixed compact subset under the geodesic flow is comparable to $\log(\omega(y))$
 (see Lemma \ref{lem:one-return}).

\medskip
\noindent{\bf Return lemma.}  We use the notation of section \ref{sec:height-func}. 

Recall that $\Lie(G)=i\mathfrak{sl}_2(\mathbb R)\oplus \mathfrak{sl}_2(\mathbb R)$.
We define a norm $\|\cdot\|$ on $\op{Lie}(G) $ using an inner product with respect to which
$\mathfrak{sl}_2(\mathbb R)$ and $i\mathfrak{sl}_2(\mathbb R)$ are orthogonal to each other.
Given a vector $w\in\Lie(G)$, we write 
\[
w=i {\rm Im}(w)+{\rm Re}(w)\in i\mathfrak{sl}_2(\mathbb R)\oplus \mathfrak{sl}_2(\mathbb R).
\]
Since the exponential map $\Lie (G) \to G$ defines a local diffeomorphism, there exists an absolute constant
 $c_1\ge 2$ satisfying the following two properties: 
 \begin{enumerate}
     \item 
for all $x\in X$, and all $w=i{\rm Im}(w)+{\rm Re}(w)\in \op{Lie}(G)$ with $\|w\| \le \max(1, \e_X)$,
\be\label{distance}
c_1^{-1}\|w\|\leq d(x,x\exp(i {\rm Im}(w))\exp({\rm Re}(w)))\leq c_1\|w\|;
\ee
\item If $d(x,x')\leq \e_X/c_1$, then $x'=x\exp(i {\rm Im}(w))\exp({\rm Re}(w))$
for some $w\in\Lie(G)$. 
 \end{enumerate}

We choose an absolute constant $d_X\ge 24$ 
so that $$X_{\e_X}\subset \{x\in X_0: \omega(x)\le d_X\}.$$

Let $\consta\label{D:K}:=\ref{D:K}(Y)$ be given by
\be\label{Done} \ref{D:K}= c_1 \alpha \left(\tfrac{6{b_1}}{\kappa \eta_0} +d_X \right) \ee
where $\kappa$ is defined by $\hat{b}_0 \mathsf p_Y^{\ds} \kappa^{\ds/2}=1/2$, $0<\rdt <1$ is as in~\eqref{eq:reduction-theory},
$\alpha\ge 1$ is as in \eqref{eq:alpha}, and $c_1$ is as in~\eqref{distance}.
We note that by increasing $\hat b_0$ if necessary, we may and will assume that $\kappa\in(0,1)$. Moreover we put $\eta_0=\frac12$ when $Y$ is convex cocompact. 

Define
\be\label{eq:def-K}
\mathcal K_Y=\{y\in Y_0: \omega(y)\leq \ref{D:K} /(c_1\alpha)\}.
\ee
Note that $X_{\e_X}\cap Y_0 \subset \mathcal K_Y$.

 The choices of the above parameters are motivated by our applications in the following lemmas. Indeed the choice of $\kappa$ is used in~\eqref{eq:use-kappa}. The multiplicative parameter $c_1\alpha$, which features in the definitions of $\ref{D:K}$ and 
$\mathcal K_Y$, is tailored so that we may utilize Lemma~\ref{AL} in the proof of Lemma~\ref{lem:volume-bd}.

\begin{lem}[Return lemma] \label{lem:one-return}
For every $y\in Y_0$, there exists some $|r|\leq 1$ 
so that $yu_ra_{-t}\in\mathcal K_Y$ where $t=\log(\rdt\omega(y)/6)$.
\end{lem}

\begin{proof}
Let $y\in Y_0-\mathcal K_Y$. By the definition of $\omega$, there exist $1\le i\le \ell$ and $g\in \tilde{\mathfrak h}_i$ so that $y=[g]$ and 
\[
\omega(y)=\omega_i(y),
\]
see \S\ref{sec:height-func} for the notation. Set $v:= v_i g$. Then 
$$\|v\|^{-1}=\omega_i(y)=\omega(y).$$

Let us write $v=w+se_3$ where $w\in V$ and $s\in\bbr$. Recall from Lemma~\ref{BONE} that
there exists $b_1>1$ so that 
\be\label{eq:BONE}
\|w\|\geq b_1^{-1}\|v\|.
\ee

Let $\kappa>0$ be as used in \eqref{Done}. Then \eqref{eq:muz-D+} implies that
\be\label{eq:use-kappa}
\mu_{y}(D^+(\tfrac{w}{\|w\|},\kappa))\leq \tfrac{1}{2}\mu_y([-1,1]).
\ee
Therefore, there exists $r\in {\rm supp}(\mu_y)\cap \Bigl([-1,1]\setminus D^+(\tfrac{w}{\|w\|},\kappa)\Bigr)$.
This means that $yu_r\in Y_0$, moreover, we have, using \eqref{eq:BONE},
\begin{align*}
\|p^+(vu_{r})\|&=\|p^+(wu_{r})\|>\kappa\|w\| \geq \kappa b_1^{-1}\|v\|.
\end{align*}

Set $t:=\log(\rdt \omega(y)/ 6)$.
Then 
\begin{align*}
\kappa b_1^{-1}{\|v\|}  \cdot\tfrac{ \eta_0 \omega(y)}{6} &= \kappa b_1^{-1}{\|v\|} e^t \le  \|p^+(vu_r)a_t\| \\
&\le \| v u_{r} a_t\|\le \|vu_r\|e^t\le 2\|v\| \cdot\tfrac{ \eta_0 \omega(y)}{6},
\end{align*} 
where we use $\|vu_r\|\leq 2\|v\|$ in the last inequality.

Hence, using the fact that $\omega(y)=\|v\|^{-1}$,
$$\tfrac{\kappa b_1^{-1}\eta_0}{6 }\le \| v u_{r} a_t\|=\|v_i g u_ra_t\|\le \tfrac{\eta_0}3.$$
This in particular implies that $gu_ra_t\in\tilde{\mathfrak h}_i$.
By Lemma \ref{ETA}, whenever $\gamma\in \Gamma$ and $1\leq j\leq \ell$ satisfy that $\tilde{\mathfrak h}_j\neq \gamma \tilde{\mathfrak h}_i$,  we have   
\[
\| v_j \gamma g u_ra_t\| \ge  \rdt;
\]
note that $i=j$ is allowed.

This and the above upper bound thus imply
 $$\omega(yu_ra_t) = \|v_i g u_ra_t\|^{-1}.$$
Therefore, 
$$\omega (yu_ra_t) \le \frac{6b_1}{\kappa \eta_0}\le  D_1/(c_1\alpha)$$
proving the claim.
\end{proof}

\noindent{\bf Number of nearby sheets.}
Recalling that $\mathfrak{sl}_2(\mathbb C)=\mathfrak{sl}_2(\mathbb R) \oplus i\mathfrak{sl}_2(\mathbb R)$,
we set $V=i\mathfrak{sl}_2(\mathbb R)$ and consider the action of $H$ on $V$ via the adjoint representation; so
$v\cdot h= h^{-1} v h$ for $v\in V$ and $h\in H$. We use the relation $g (\exp v) h= gh \exp (v\cdot h)$ which is valid for all 
$g\in G, v\in V, h\in H$.

If $D\ge \alpha/2$ for $\alpha$ as in Proposition \ref{alpha}, then $D^{-1}\omega(y)^{-1}\le \tfrac12  \inj (y)$.

\begin{Def} \label{zzzii}For $y\in Y_0$ and $D\ge \alpha/2$, we define 
\be\label{ii}
I_Z(y,D)=\{v\in V-\{0\}: \|v\|< D^{-1}\omega(y)^{-1},\; y\exp(v)\in Z\}.
\ee
\end{Def}
Since $V$ is the orthogonal complement to $\Lie (H)$,
the set $I_Z(y, D)$  can be understood as the number of sheets of $Z$ in the ball around $y$ of radius $D^{-1}\omega(y)^{-1}$.

It turns out that $\# I_Z(y, D)$ can be controlled  in terms of the tight area of $S_Z$, uniformly over all $y\in Y_0$ for an appropriate $D>1$.

\begin{Not} \label{TI} We set
 \[
 \tau_Z :=\area_t(S_Z).
 \]
\end{Not}
Theorem ~\ref{lem:VZ-finite1} shows that $1\ll \tau_Z<\infty$ where the implied constant depends only on $M$.

We begin with the following lemma:

\begin{Lem} \label{AL} With $c_1\ge 2$ and $\alpha\ge 2$ given respectively  in \eqref{distance} and \eqref{alpha},
we have that for all $y\in Y_0$,
\be\label{k} \# I_Z(y, c_1 \alpha) \ll \omega(y)^{3} \tau_Z .\ee

\end{Lem}

\begin{proof} Let $c_1\ge 1$ and $\alpha$ be the absolute constants given in \eqref{distance} and \eqref{alpha}
respectively. It follows that for  any $y\in Y_0$ and $v\in I_Z(y,\alpha)$, 
\be\label{eq:inj-omega-E0}
d(y,y\exp(v))\leq c_1\|v\| \leq  c_1(c_1\alpha)^{-1} \cdot  \omega(y)^{-1} <\tfrac{1}{2}  \cdot {\op{inj}(y)}.
\ee

It follows that
 for each $v\in I_Z(y, c_1\alpha)$, $\inj(y\exp v)\ge {\inj (y)}/{2}$. Hence the balls $B_Z(y\exp v, {\op{inj}(y)}/{2})$, $v\in I_Z(y,c_1 \alpha)$ are disjoint from each other, and hence
\[
\# I_Z(y, \alpha) \cdot  \op{Vol}(B_H(e, {\op{inj}(y)}/{2}))
=\op{Vol}\Bigl\{\bigcup B_Z(y\exp v, {\op{inj}(y)}/{2}) :v\in I_Z(y, \alpha)\Bigr\}. 
\]

On the other hand, if we set $\rho_y:=\min\{1, \inj(y)/2\}$, then
$$\pi\Bigl(\bigl\{\bigcup B_Z(y\exp v, \rho_y) :v\in I_Z(y, c_1\alpha)\bigr\}\Bigr)\subset S_Z\cap\mathcal N(\core(M)).$$
Therefore
$$ \# I_Z(y, c_1\alpha) \le \op{Vol}(B_H(e, \rho_y))^{-1} \cdot \tau_Z\ll \rho_y^{-3}\tau_Z\ll \omega(y)^{3} \tau_Z ;$$
we have used that
$2\pi(\cosh r -1)\ge r^3$ for all $r>0$ and Proposition \ref{alpha} respectively in the last two estimates.
\end{proof}

Let $D_1$ be as in~\eqref{Done}. By the choice of $\kappa$,  we have $D_1\ll\mathsf p_Y^2$
(see the discussion following~\eqref{Done}).

 \begin{lem}[Number of sheets] \label{lem:volume-bd}
For $D_1=D_1(Y)\ll \mathsf p_Y^2$ as in \eqref{Done}, we  have
\[
\sup_{y\in Y_0} \#I_Z(y,\ref{D:K})\le  c_0\cdot  {\mathsf p}_Y^{6} \cdot \tau_Z
\]
where $c_0\geq 2$ is an absolute constant.
\end{lem}
\begin{proof} 
Let $\mathcal K_Y$ be as in \eqref{eq:def-K}:
\begin{equation*}
\mathcal K_Y=\{y\in Y_0: \omega(y)\leq (c_1\alpha)^{-1} \ref{D:K}\}.
\end{equation*}

If $y\in \mathcal K_Y$, then, by Lemma \ref{AL},
 $$\#I_Z(y, \ref{D:K}) \le \#I_Z(y,c_1\alpha)\ll D_1^3 \tau_Z\ll \mathsf p_Y^6\tau_Z .$$
Now suppose that $y\in Y_0-\mathcal K_Y$. 
By Lemma \ref{lem:one-return}, there exist $|r|<1$ and $t=\log (\rdt\cdot \omega(y)/6)$, where $0<\rdt\leq 1$ is as in~\eqref{eq:reduction-theory}, such that
\[
yu_ra_t \in \mathcal K_Y.
\]

We claim that if $v\in I_Z(y,\ref{D:K})$, then 
$v(u_{r}a_{t}) \in  I_Z(yu_ra_t, c_1\alpha)$.
Firstly, note that, plugging $t=\log (\rdt\cdot \omega(y)/6)$ and using $0<\eta\le 1$,
\[
\|v(u_ra_{t})\| \leq 3 e^t \|v\| = \tfrac{3\rdt\, \omega(y)\,  \|v\|}{6} <  \omega(y) \cdot  \|v\|.
\]

Hence for $v\in  I_Z(y,\ref{D:K})$,  as $\omega(y) \|v\|< D_1^{-1}$,
\[
\|v (u_ra_{t})\|<  \omega(y) \cdot \|v\|\le  D_1^{-1}\leq  (c_1\alpha)^{-1}\omega(yu_ra_t)^{-1}. 
\]
where we used the fact that $(c_1\alpha)^{-1}\ref{D:K} > \omega(yu_ra_t)$. 

Since 
$y(\exp v) u_{r}a_{t} = (y u_ra_t)\exp (v(u_{r}a_t)) \in Z$,
 this implies that
$v(u_{r}a_{t}) \in  I_Z(yu_ra_t, c_1\alpha)$. Therefore the map $v\mapsto v(u_{r}a_{t})  $ 
is an injective map from $I_Z(y,\ref{D:K})$ into $I_Z(yu_ra_t, c_1\alpha)$.
Consequently, \[
\#I_Z(y, \ref{D:K})\le \#I_Z(yu_ra_t, c_1\alpha) \ll {\mathsf p}_Y^{6} \cdot \tau_Z.
\] 
This finishes the proof.
\end{proof}

\section{Margulis function: construction and estimate}\label{sec:isolation-Marg}
Throughout this section, we fix closed non-elementary $H$-orbits $Y, Z$ in $X$ and 
\[
\frac{\ds}{3}\le s< {\ds}.
\] 

In this section, we define a family of Margulis functions $F_{s, \lambda}=F_{s,\lambda, Y, Z}$, $\lambda>1$
and show that the hypothesis of Proposition \ref{prop:iteration} is satisfied for a certain choice of $\lambda$, 
which we will denote by $\lambda_s$. As a consequence, we will get an estimate on $m_Y(F_{s, \lambda_s})$ in Theorem
\ref{thm:limit-thick}.

We set
$$I_Z(y):=\{v\in V-\{0\}: \|v\|< D_1^{-1}\omega(y)^{-1},\; y\exp(v)\in Z\}$$ for $D_1>1$  as given in Lemma~\ref{lem:volume-bd}.

\begin{Def}[Margulis function] \label{eq:def-f}\label{MFF} \rm

\begin{enumerate}
\item Define $f_s:=f_{s,Y, Z}:Y_0 \to (0,\infty)$ by
\begin{equation*}f_s(y):=\begin{cases}\sum_{v\in I_Z(y)}\|v\|^{-s} & \text{if $I_Z(y)\ne \emptyset$}\\
 \omega(y)^{s} &\text{otherwise}.\end{cases}
\end{equation*}

\item For $\lambda\ge 1$, 
 define $F_{s,\lambda}=F_{s, \lambda, Y, Z}:Y_0\to (0,\infty)$ as follows:
\be\label{eq:def-F}
F_{s, \lambda}(y)=f_{s}(y)+\lambda \;\omega (y)^s.
\ee
\end{enumerate}
 \end{Def}

Note that  for all $y\in Y_0$
\be\label{min} 
 \omega(y)^{s} \le f_s(y)<\infty.
\ee 
Since $Y$ and $Z$ are closed orbits, both $f_s$ and $F_{s,\lambda}$ are locally bounded.
Moreover, they are also Borel functions. Indeed, $\omega^s$ is continuous on $Y_0$, and
$f_s$ is continuous on the open subset $\{y\in Y_0: I_Z(y)\neq \emptyset\}$ as well as on its complement.

 In this section, we specify choices of parameters $t_s$ and $\lambda_s$ so that
 the average ${\mathsf A}_{t_s}F_{s, \lambda_s}$ satisfies the hypothesis of Proposition ~\ref{prop:iteration}  with controlled size of the additive term
 (Lemma \ref{lem:ineq-U-F-local}).

 \begin{Not}[Parameters] \label{TAU} 
 \begin{enumerate}
 \item For $0<c <1 $, define $t(c, s)>0$ by
$$\frac{ b_0 b_1 {\mathsf p}_Y^{\ds} e^{-(\ds -s)t(c,s) /4} }{(\ds-s)}= c $$
where $b_0$ and $b_1$ are given in Lemma \ref{la}.
\item For $0<c<1$ and $t>0$,  define $\lambda(t,c,s)>0$ by
$$\lambda(t,c,s):=\left( 2 c_0  D_1 p_Y^{6}\tau_Z\right) \frac{ e^{2ts}}{c} $$
where $c_0$ is given by \eqref{lem:volume-bd}.
\end{enumerate}
\end{Not}

As it is evident from the above, the definition of $t(c,s)$ is motivated by the linear algebra lemma \ref{la}.
Indeed, for any  vector $v\in e_1G$ and $t\ge t(c,s)$, we have 
we have
\begin{align}\label{ANY}
\sup_{1\le \rho\le 2}
\frac{1}{\mu_y[-\rho,\rho]}\int_{-\rho}^{\rho} \frac{1}{\| vu_r a_t\|^{s}} d\mu_{y}{} (r)\le  {c\|v\|^{-s}}.
\end{align}

The choice of $\lambda(t,c,s)$ is to control the additive 
 difference  between $f_s(yu_ra_t ) $ and $\sum_{v\in I_Z(y) }\|vu_ra_t\|^{-s}$ uniformly over all $r\in [-1,1]$ such that $yu_r\in Y_0$,
so that we will get:
$$
{\mathsf A}_{t} f_{s}(y) \le  c\cdot f_{s}(y)+ \tfrac{\lambda (t,c,s)  c}{2} \omega(y)^{s}
$$
(see Lemma \ref{QQ}, \eqref{eq:before-la} and \eqref{eq:average-f-s'-F}).

\medskip

\noindent{\bf Markov operator for the height function.}\label{sec:isolation}
In this subsection, we use notation from section \ref{sec:height-func}.

It will be convenient to introduce the following notation:
 \begin{Not}\label{sec:Not}\rm
  Let $Q\subset G$ be a compact subset.
\begin{enumerate}
\item Let  $d_Q\geq 1$ be the infimum of all $d\geq 1$ such that
 for all $g\in Q$ and $v\in \br^4$,
  \be\label{ww2} d^{-1} \|v\| \le \|  vg \|\le  d \|v\|.\ee
Note that $d_Q\asymp \max_{g\in Q} \|g\|$, up to an absolute multiplicative constant.

\item We also define $c_Q\geq 1$ to be the infimum of all $c\geq 1$ such that  for any $x\in X_0$, $g\in Q$ with $xg\in X_0$, and 
for all $1\le i \le \ell$
\be\label{ww1} 
c^{-1} \omega_i(x) \le \omega_i(xg)\le c\, \omega_i(x).
\ee
We note that
$c_Q\asymp \max _{g\in Q} \|g\|$ up to an absolute multiplicative constant.
\end{enumerate}
\end{Not}

 \begin{lem}\label{lem:ineq-cusp-local}
For any $ 0<c\leq 1/2$ and $t \ge t(c, s)$, there exists
$\consta\label{D:D-omega}\asymp e^{2t}$ so that 
for all $y\in Y_0$ and $1\leq \rho\leq 2$,
\[
{\mathsf A}_{t,\rho} \omega (y)^s\le c \cdot \omega(y)^s+\ref{D:D-omega}.
\]

\end{lem}

\begin{proof}
Let $t\geq t(c, s)$. We compare $\omega(yu_ra_t)$ and $\omega(y)$ for $r\in[-2,2]$.
Setting $$Q:= \{a_\tau u_r :|r|\leq 2, |\tau|\leq t\},$$
we have 
$c_Q \asymp e^t $.

Let $\eta_0$ be as in Lemma~\ref{ETA}. Fix $0<\eta_X\leq  \min\{\e_X,\eta_0\}$ so that 
\[
\eta_X\asymp \e_X\quad\text{and}\quad\eta_X^{-1}\geq \sup_{y\in X_{\e_X}\cap Y_0} \omega(y);
\]

We consider two cases. 

\noindent{{\bf Case 1}: $\omega(y)\leq 2c_Q/\eta_X$.}
 In this case, for $h\in Q$ with $yh\in Y_0$,
\[
\omega(yh)\leq 2c_Q^2/\eta_X.
\]
 Hence, the claim in this case follows if we choose $\ref{D:D-omega}=2c_Q^2/\eta_X\asymp e^{2t}$.

\noindent{{\bf Case 2}: $\omega(y) >2c_Q/\eta_X$.}
By the definition of $\omega$, there exists $1\le i\le \ell$ such that
 \[
 \omega_i(y)>2c_Q/\eta_X,\quad\text{ and hence}\quad y\in \mathfrak h_{i}.
 \] 
By the definition of $c_Q$, see~\eqref{ww1}, we have
\[
\omega_i(yh)>2/\eta_X,\quad\text{and hence}\quad yh\in \mathfrak h_{i}
 \] 
for all $h\in Q$ with $yh\in Y_0$. 
Choose $g_0\in G$ so that $y=[g_0]$. In view of Lemma~\ref{ETA}, see in particular~\eqref{eq:reduction-theory}, and since $\eta_X\leq \eta_0$
there exists $\gamma\in\Gamma$ such that simultaneously for all $h\in Q$ with $yh\in Y_0$,
\[
\omega(yh)=\omega_i(yh)=\|v_{i}\gamma g_0h\|^{-1}.
\]
Since $v_i=e_1g_i^{-1} \in e_1G$ (see \eqref{vi}), we may apply  Lemma \ref{la} (linear algebra lemma II) and deduce:
\begin{align*}
{\mathsf A}_{t,\rho}\omega(y)^s&=\frac{1}{\mu_y([-\rho,\rho])}\int_{-\rho}^{\rho} \frac{1}{\| v_{i}\gamma u_r a_t \|^s} d\mu_y(r)\\
&\leq   \frac{ b_0b_1{\mathsf p}_Y^{\ds} e^{-(\ds-s) t/4 } }{(\ds-s)} \|v_{i}\gamma\|^{-s}\le c\cdot \omega(y)^s;
\end{align*}
in the last inequality we used the fact that $t\geq t(c,s)$. The proof is now complete.
\end{proof}

\medskip

\noindent{\bf Log-continuity of $F_{s,\lambda}$.}
The following log-continuity lemma with a control on the multiplicative constant $\sigma$ is the first hypothesis in Proposition \ref{prop:iteration}.
 \begin{lem}[Log-continuity lemma]  \label{F}
 There exists  $2\le \sigma\ll {\mathsf p}_Y^{8}$ so that the following holds:
 for every $\lambda \geq \tau_Z$,  we have 
\[
\sigma^{-1} F_{s,\lambda} (y) \le F_{s,\lambda}(yh) \le \sigma F_{s,\lambda}(y)
\]
for all $y\in Y_0$ and all $h\in B_H(2)$ so that $yh\in Y_0$. 
\end{lem}

Let $c_0$ be as in Lemma~\ref{lem:volume-bd}. Recall from Theorem ~\ref{lem:VZ-finite1} that $\tau_Z\geq \e_X^2$, replacing $c_0$ by its multiple (which we continue to denote by $c_0$) if necessary we assume that $c_0\tau_Z\geq 1$.

We first obtain  the following estimate for $f$ on nearby points:
 \begin{lem}\label{QQ}
 Let $Q\subset H$ be a compact subset. For any $y\in Y_0$ and $h\in Q$ such that $yh\in Y_0$, we have 
 \[
 f_{s}(yh)\le \sum_{v\in I_Z(y)} \|vh \|^{-s} +{}\left( {c_0} c_Q d_Q D_1 {\mathsf p}_Y^{6} \tau_Z \right)  \omega(y)^s
 \]
 where $c_0$ is as above and
 the sum is understood as $0$ when $I_Z(y)=\emptyset$.
 \end{lem}

 \begin{proof} Let $y\in Y_0$ and $h\in Q$ with $yh\in Y_0$.
 If $I_Z(yh)=\emptyset$, then by \eqref{ww1}, we have
 $$ f_{s}(yh)= \omega(yh)^{s} \le c_Q ^s \omega(y)^{s} $$
proving the claim; recall that $c_0\tau_Z\geq 1$. 

Now suppose that $I_Z(yh)\ne  \emptyset$.
Setting $$\e:=( d_Q D_1 \omega(y))^{-1},$$
we write
\begin{align}
\label{use} f_{s}(yh)
= \sum_{v\in I_Z(yh), \|v\|< \epsilon}\|v\|^{-s}+\sum_{v\in I_Z(yh), \|v\|\geq \epsilon}\|v\|^{-s} .
\end{align}

Since $\# I_Z(yh)\le c_0 {\mathsf p}_Y^{6} \tau_Z$ by Lemma~\ref{lem:volume-bd}, we have
\begin{align}
\label{use2} \sum_{v\in I_Z(yh), \|v\|\geq \epsilon}\|v\|^{-s}&\leq \bigl(c_0 {\mathsf p}_Y^{6}\tau_Z\bigr)\epsilon^{-s}\le \bigl(c_0  d_Q
D_1 {\mathsf p}_Y^{6}\tau_Z \bigr)\omega(y)^{s}.
\end{align}

Thus, if there is no $v\in I_Z(yh)$ with $\|v\|\le \e$, then the lemma follows from \eqref{use}. 

If $v\in I_Z(yh)$ satisfies $\|v\|< \epsilon$, then 
\[
\|vh^{-1} \|\leq d_Q \epsilon=D_1^{-1} \omega(y)^{-1};
\] 
in particular, $vh^{-1}\in I_Z(y)$.
Therefore, by setting $v'=vh^{-1}$, 
$$\sum_{v\in I_Z(yh), \|v\|< \epsilon}\|v\|^{-s}\le \sum_{v'\in I_Z(y)}\|v' h\|^{-s}.$$
Together with \eqref{use2}, this finishes the proof.
\end{proof}

\noindent{\bf Proof of Lemma \ref{F}}.
Since $B_H(2)^{-1}=B_H(2)$, it suffices to show the inequality $\le$.
By Lemma \ref{QQ}, applied with $Q=B_H(2)$, $c:=c_{B_H(2)}$ and $d:=d_{B_H(2)}$,
we have that for all $h\in B_H(1)$ with $yh\in Y_0$, we have 
\begin{align*}
f_{s}(yh)&\le \sum_{v\in I_Z(y)} \|vh \|^{-s} +{}\left( {c_0} c d D_1 {\mathsf p}_Y^{6} \tau_Z \right)  \omega(y)^s \\
&\leq  d   \sum_{v\in I_Z(y)} \|v\|^{-s} +{}c_0c d   D_1 {\mathsf p}_Y^{6}\tau_Z  \omega(y)^s. 
\end{align*}
where we used the definition of $d$. 

Recall from Theorem ~\ref{lem:VZ-finite1} that $\e_X^2\leq \tau_Z\leq \lambda$ and that $D_1\ll \mathsf p_Y^2$.

If $I_Z(y)=\emptyset$, then
\begin{align*}
F_{s,\lambda}(yh)&\ll {\mathsf p}_Y^8 \tau_Z\omega (y)^s+\lambda \omega(y)^s\ll \mathsf p_Y^{8} \lambda \omega(y)^s\\
&\ll {\mathsf p}_Y^{8} (f_{s}(y) + \lambda \omega(y)^s)\ll{\mathsf p}_Y^{8} F_{s,\lambda}(y).
\end{align*}

If $I_Z(y)\ne \emptyset$, then \begin{align*}
F_{s,\lambda}(yh)&\leq d\cdot  f_{s}(y)+ {}c_0c d   D_1 {\mathsf p}_Y^{6}\tau_Z  \omega(y)^s  +\lambda\omega(yh)^s\\
&\ll f_{s}(y)+\mathsf p_Y^{8} \lambda\omega(y)^s \ll   \mathsf p_Y^{8} F_{s,\lambda}(y).
\end{align*}
This finishes the upper bound. The lower bound can be obtained similarly.

\medskip

\noindent{\bf Main inequality.}
We will apply the following lemma to obtain the second hypothesis
of Proposition \ref{prop:iteration} for $c:=(8\sigma \mathsf p_Y^{\ds})^{-1}<1/2$.

\begin{lem}[Main inequality]\label{lem:ineq-U-F-local}
Let $ 0<c\leq1/2$. For $t\geq t(c/2,s)$ and $\lambda=\lambda(t,c,s) $, we have the following:
for any
$y\in Y_0$ and $1\leq\rho\leq 2$, we have
\[
{\mathsf A}_{t,\rho} F_{s,\lambda} (y)\le c \, F_{s,\lambda}(y)+\lambda\ref{D:D-omega}
\]
where 
$\ref{D:D-omega} \ll e^{2t}$ is as in Lemma~\ref{lem:ineq-cusp-local}.
\end{lem}

\begin{proof}
The following argument is based on comparing the values of
$f_{s}(yu_ra_t)$ and $f_{s}(y)$ for $r\in [-2,2]$ such that $yu_ra_t\in Y_0$.

Let $Q:=\{a_\tau u_r: |r| \le 2 , |\tau|\leq t\}$. Then
\[
c_Q {\asymp}e^t\quad\text{ and }\quad d_Q {\asymp}e^{t}
\] 
 where $c_Q$ and $d_Q$ are as in~\eqref{sec:Not}.
Hence, by Lemma \ref{QQ}, we have {that} for any $|r|\le 2$ such that $yu_ra_t\in Y_0$,
\be\label{eq:before-la}
f_{s}(yu_ra_t)\le \sum_{v\in I_Z(y)} \|v u_ra_t\|^{-s} + c_0 D_1   {\mathsf p}_Y^{6} \tau_Z \omega(y)^s e^{2ts}
\ee
where $c_0$ is as in Lemma~\ref{QQ}.

By averaging ~\eqref{eq:before-la} over $[-\rho,\rho]$ with respect to $\mu_y$, 
and applying \eqref{ANY}, we get
\begin{align}
\label{eq:average-f-s'-F}
{\mathsf A}_{t,\rho} f_{s}(y)&\leq c\cdot f_{s}(y)+ c_0  D_1 {\mathsf p}_Y^{6} \tau_Z \omega(y)^{s} e^{2ts} \\
\notag&\leq c\cdot f_{s}(y)+ \tfrac{\lambda c}{2} \omega(y)^{s}.
\end{align}

Then by Lemma \ref{lem:ineq-cusp-local} and~\eqref{eq:average-f-s'-F},
we have 
\begin{align*}
{\mathsf A}_{t,\rho} F_{s,\lambda}(y) & = {\mathsf A}_{t,\rho} f_{s}(y)+ {\mathsf A}_{t,\rho} \lambda\omega(y)^s
\\ &\leq c\cdot f_{s}(y)+\tfrac{c\lambda}{2}\omega(y)^{s}+\tfrac{c\lambda}{2}\omega(y)^{s}+\lambda\ref{D:D-omega}\\
&=c\cdot F_{s,\lambda}(y)+\lambda\ref{D:D-omega}.
\end{align*}
\end{proof}

By Theorem \ref{spp}, we have $\mathsf s_Y\asymp {\mathsf p}_Y$.
For the sake of simplicity of notation, we put
\be\label{alpha3} \alpha_{Y,s}:=\left(\frac{{\mathsf s}_Y}{ \ds-s}\right) ^{1/ (\ds-s)}
\asymp \left(\frac{{\mathsf p}_Y}{ \ds-s}\right) ^{1/ (\ds-s)}.\ee

 We are now in a position to apply Proposition \ref{prop:iteration} to get the following estimate:
\begin{thm}[Margulis function on average] \label{thm:limit-thick} 
There exists $\lambda_s>1$ such that
$$
m_Y(F_{s, \lambda_s}) \ll  \alpha_{Y,s}^\star \tau_Z. 
$$

\end{thm}

 \begin{proof}
 
 Let $1\leq \sigma\ll \mathsf p_Y^8$ be given by Lemma~\ref{F}.
Let  $c:=(8\sigma \mathsf p_Y^{\ds})^{-1}<1/2$, $t_s:= t(c, s)$ and $\lambda_s:=\lambda( t_s,c, s)$ be given by \eqref{TAU}.
Then in view of Lemmas~\ref{F} and~\ref{lem:ineq-U-F-local}, 
$F_{s,\lambda_s}$ satisfies the conditions of Proposition~\ref{prop:iteration}
with $t=t_s$ and $D_0=\lambda_s \ref{D:D-omega}$,
where $D_2\ll e^{2t_s}$ is given in Lemma \ref{lem:ineq-cusp-local}.
Therefore 
\be\label{myy}m_Y(F_{s, \lambda_s})\le 64\lambda_s \mathsf p_Y^{\ds}D_2. \ee

Since 
$$  e^{(\ds -s)t_s } =\tfrac{ (8\sigma b_0 b_1 {\mathsf p}_Y^{2\ds})^4  }{(\ds-s)^4} \ll \left(\tfrac{  {\mathsf p}_Y  }{\ds-s} \right)^{\star} \quad\text{and}
\quad \lambda_s=\left( 2 c_0  D_1 p_Y^{6}\tau_Z\right) \tfrac{ e^{2t_s s}}{c},$$
we get  
\[
\lambda_s \mathsf p_Y^{\ds}D_2  \ll \mathsf p_Y^\star e^{4t_s} \tau_Z\ll\alpha_{Y,s}^\star  \tau_Z.
\]
Combining this with \eqref{myy} finishes the proof.
 \end{proof}

 \section{Quantitative isolation of a closed orbit} \label{final2}
In this section, we deduce Theorem \ref{main} from Theorem \ref{thm:limit-thick}.
 Let $Y, Z$ be non-elementary closed $H$-orbits in $X$.  We allow the case $Y=Z$ as well.
  Let $\frac{\ds}{3} \le s< \ds$.
  
Recall the definitions of $f_s=f_{s,Y, Z}$ and $F_{s,\lambda}=F_{s,\lambda, Y, Z}$ from Definition \ref{MFF}. Let $\lambda_s$ be given by
 Theorem \ref{thm:limit-thick}. Using the log-continuity lemma for $F_{s,\lambda_s}$ (Lemma \ref{F}), we first deduce the following estimate:
 \begin{prop}\label{thm:isolation-delta}  For any $0<\epsilon<\e_X$ and $y\in Y_0\cap X_\epsilon$, we have 
\[
f_{s, Y, Z}(y)\le F_{s,\lambda_s}(y)\ll \frac{\alpha_{Y, s}^\star \tau_Z }{m_Y(B(y, \e))}.
\]
\end{prop}
 
 \begin{proof}    Let $y\in Y_0\cap X_\e$. Then $\inj(y)\ge \e$ and hence $y B_H(\e)= B(y, \e)$. 
For all $h\in B_H(\e_X)$, $F_{s,\lambda_s} (y) \le \sigma F_{s,\lambda_s}(yh)$ for some constant
 $\sigma\ll \mathsf p_Y^{6}$ by Lemma \ref{F}.
By applying Theorem ~\ref{thm:limit-thick}, we get
\begin{align*}
F_{s, \lambda_s} (y) \le \frac{ \sigma \int_{x\in y B_H(\e)} F_{s,\lambda_s}(x) dm_Y(x)}{m_Y(B(y, \e))}\le \frac{ \sigma \cdot m_Y(F_{s,\lambda_s})}{m_Y(B(y, \e))}
\ll \frac{\alpha_{Y, s}^\star  \tau_Z }{m_Y(B(y, \e))}.
\end{align*}
\end{proof}

Recall from~\eqref{eq:alpha} that for all $x\in X_0$,
\be\label{eq:alpha-1}
\tfrac{1}{2\alpha} \cdot \inj (x)\le  \omega(x)^{- 1} \le  \tfrac{\alpha}{2} \cdot \inj (x).
\ee

Using the next lemma, we will be able to use the estimate for $f_{s, Y, Z}$ 
obtained in Proposition \ref{thm:isolation-delta} to deduce a lower bound for $d(y, Z)$.

\begin{lemma}\label{Final}
\begin{enumerate}  \item Let $y\in Y_0$ and $z\in Z- B_Y(y, \inj(y))$.
 If 
$ d(y, z) \le \frac{1}{2\alpha c_1D_1}\inj(y)$, 
 then $$d(y, z)^{-s}\le c_1 f_{s, Y, Z}(y)$$
 where $c_1\ge 1$ is as in~\eqref{distance}.
 \item If $Y\ne Z$, then for any $y\in Y_0$,
 $$d(y, Z)^{-s}  \ll \mathsf p_Y^2 f_{s,Y, Z}(y).$$
\end{enumerate}
\end{lemma}
\begin{proof} 
As $Z$ is closed and  $d(y, z) \le \frac{1}{2\alpha c_1D_1}\inj (y) <\frac{1}{2} \inj(y)$, 
the hypothesis $z\in Z- B_Y(y, \inj(y))$ and the choice of $c_1$ implies that
$z$ is of the form $y\exp (v)\exp(v')$ with 
$v\in i \mathfrak{sl}_2(\mathbb R)-\{0\}$ and $v'\in \mathfrak{sl}_2(\mathbb R)$.

In particular $y\exp(v)=z\exp(-v')\in Z$. Moreover, by \eqref{distance},
$$ \|v\|\le \|v+v'\|\le c_1 d(y, z) \le  D_1^{-1}\inj (y)/(2\alpha) \le (D_1 \omega(y))^{-1}.
$$
It follows that $v\in I_Z(y, D_1)$.
Therefore 
\be\label{fss}
d(y, z)^{-s} \le c_1^s\|v\|^{-s}\le c_1 \|v\|^{-s} \le c_1 f_{s}(y),
\ee
proving (1).

We now turn to the proof of~(2); suppose thus that $Y\ne Z$. 
Then there exists $z\in Z$ such that
$d(y, Z)=d(y,z)$. In view of (1), it suffices to consider the case when
$d(y, z)>  \frac{1}{2\alpha c_1D_1}\inj(y)$.

Since $s\leq 1$, $\omega(y)^s\le f_s(y)$, and $D_1\ll \mathsf p_Y^2$, we get
\[
d(y,z)^{-s} \le {2\alpha c_1D_1} \inj(y)^{-s}\le 2{\alpha^2 c_1D_1}\omega(y)^s \ll  \mathsf p_Y^2 f_{s, Y, Z} (y)
\]
where we also used~\eqref{eq:alpha-1}. The proof is complete.
\end{proof}

Theorem \ref{main}(1) is a special case of the following theorem:
\begin{thm}[Isolation in distance] \label{cor:main-dist}
For any $0<\epsilon<\e_X$, $y\in Y_0\cap X_\epsilon$, and $z\in Z$, at least one of the following holds:
\begin{enumerate}
\item $z\in  B_Y(y,\e)=yB_H(e,\e)$, or
\item 
$d(y,z)\gg \alpha_{Y,s}^{-\star/s}  {m_Y(B(y, \e))}^{1/s} {\tau_Z}^{-1/s}$,
where $\alpha_{Y,s}$ is as given in \eqref{alpha3}.

\end{enumerate}
\end{thm}
\begin{proof} As $y\in X_\e$, $\inj(y)\ge \e$.
Suppose that  $z\notin B_Y(y, \e)$.  We first observe that  since
 $m_Y(B(y, \e))^{1/s}\ll \e$ and $\mathsf p_Y^{-2}\gg  \alpha_{Y,s}^{-\star/s} $, we have
  $$\frac{\e}{2\alpha c_1D_1}\gg \mathsf p_Y^{-2}\e \gg  \alpha_{Y,s}^{-\star/s}  {m_Y(B(y, \e))}^{1/s} .$$

Therefore, if  $d(y, z) \ge \frac{1}{2\alpha c_1D_1}\e$, then $(2)$ holds in view of the fact that $\tau_Z\ge \e_X^2$.  

 If $d(y, z) \le \frac{1}{2\alpha c_1D_1}\e\le \frac{1}{2\alpha c_1D_1}\inj (y)$, then by Lemma \ref{Final}, $d(y, z)^{-s} \le c_1 f_{s}(y)$.
Hence applying Proposition \ref{thm:isolation-delta}, we conclude
\[
d(y, z)^{-s} \le c_1 f_{s}(y)\le c_1\frac{\alpha_{Y, s}^\star \tau_Z }{m_Y(B(y, \e))}
\]
which finishes the proof in this case as well. 
\end{proof}

The following theorem is Theorem \ref{main}(2):
\begin{thm}[Isolation in measure]\label{fd-U} 
Let $0<\epsilon\leq \e_X$. Let $Y\ne Z$.
We have 
\[
m_Y\{y\in Y: d(y, Z)\le \epsilon\}\ll  \alpha_{Y, s}^{\star} \tau_Z \e^s .
\]
\end{thm}

\begin{proof}  Let $\lambda_s$ be given by
 Theorem \ref{thm:limit-thick}. By Lemma \ref{Final}(2),
$$d(y, Z)^{-s}\le c f_{s, Y. Z}(y)\le C\cdot F_{s, \lambda_s}(y)$$ for some $1<C\ll \mathsf p_Y^2$.

For $0<\e<\e_X$, if we set
\[
\Omega_{\e}:=\{y\in Y_0: F_{s,\lambda_s} (y)>C^{-1} \e^{-s}\},
\]
then 
$\{y\in Y_0: d(y, Z)\le \epsilon\} \subset \Omega_\e$.
On the other hand, we have
 \[
 C^{-1}\e^{-s} {m_{Y}(\Omega_{\e})} \le \int_{\Omega_{\e}} F_{s, \lambda_s} dm_Y
 \le m_Y(F_{s,\lambda_s}).
 \]

Since $m_Y(F_{s,\lambda_s})  \ll \alpha_{Y, s}^{\star} \tau_Z$ by Theorem \ref{thm:limit-thick},
we get that
\[
m_Y\{y\in Y_0: d(y, Z)\le \e \} \le  m_Y(\Omega_{\e}) \ll \alpha_{Y, s}^{\star} \tau_Z \e^s.
\]
\end{proof}

 \begin{proof}[Proof of Proposition~\ref{prop:main-tech}]
 Let $F_s=F_{s, \lambda_s}$ be as in Theorem~\ref{thm:limit-thick}. 
 Then $F_s$ satisfies~(1) in the proposition by Lemma~\ref{Final}.
 It satisfies~(3) by Lemma~\ref{F}. 
 
 Moreover, in view of Lemmas~\ref{F} and~\ref{lem:ineq-U-F-local}, 
$F_{s}$ satisfies the conditions of Proposition~\ref{prop:iteration}. Hence, by Proposition~\ref{prop:iteration2}, it also satisfies~(2) in the proposition.  
 \end{proof}

We remark that in both Theorems \ref{cor:main-dist} and \ref{fd-U}, the exponents $\star$ depend only on $G$, and
the implied constants are respectively of the form $c\,\e_X^{N}$ and
$c^{-1} \,\e_X^{-N}$ for some $c\le 1$ and $N\ge 1$ both depending only on $G$.

\subsection*{Number of properly immersed geodesic planes}
When $\Vol(M)<\infty$, 
we record the following corollary of Theorem \ref{cor:main-dist}.
Let $\mathcal N(T)$ denote the number of properly immersed totally geodesic  
planes $P$ in $M$ of area at most $T$.

We deduce the following upper bound from Theorem \ref{cor:main-dist}  using the pigeonhole principle: 

\begin{cor}\label{volvol} Let $\Vol (M)<\infty$. There exists $N\ge 1$ (depending only on $G$) such that for any $1/2<s<1$, we have
$$\mathcal N(T)\ll_s  \Vol (M) \e_X^{-N} T^{\frac{6}{s}-1}   $$
where the implied constant depends only on $s$.
\end{cor}

\begin{proof} We begin by recalling that $\alpha_{Y,s}=\alpha_s:=(\frac{1}{1-s})^{1/(1-s)}$ for any closed $H$-orbit $Y$ in $X$ when $\Vol (M)<\infty$.

We obtain an upper bound for the number of closed $H$-orbits in $X$ which yields the above result.
The proof is based on applying Theorem~\ref{cor:main-dist}. 

If $X$ is compact, let $\rho=0.1\e_X$. If $X$ is not compact, then the quantitative non-divergence of the action of $U$ on $X$ implies 
that there exists $\rho>0$ so that for all $x\in X$ such that $xU$ is not compact,
$$
\frac{1}{T}\ell \{t\in [0, T]: x u_t\in
X-X_\rho\} \le 0.01 $$
for all sufficiently large $T\gg 1$,  e.g., see~\cite{DM}.  Moreover $\rho$ can be taken to be $\asymp \e_X^k$ for some $k\ge 1$.

Since $(Y, m_Y)$ is $U$-ergodic by the Moore's ergodicity theorem for every closed orbit $Y=xH$,
 the Birkhoff ergodic theorem says that for $m_Y$
a.e. $y\in Y$,
$$\lim_{T\to \infty} \frac{1}{T}\ell \{t\in [0, T]: y u_t\in
X-X_\rho\} =m_Y(X-X_\rho)$$
where $\ell$ denotes the Lebesgue measure on $\br$; therefore
\be\label{eq:non-div-xH}
m_Y(X\setminus X_\rho)< 0.01.
\ee

For every $S>0$ put 
\[
\mathcal Y(S):=\{xH:\text{$xH$ is closed and $S/2< \Vol(xH)\leq S$}\}.
\] 
In view of the above choice of $\rho$, we have $\Vol(xH)\geq \rho^{3}\gg1$ for every closed orbit $xH$. 
Let $n_0=\lfloor 3\log_2(\rho)\rfloor$ and for every $T>1$, let $n_T=\lceil\log_2 T\rceil$. Then we have 
\[
\{xH: \text{$xH$ is closed and $\vol(xH)\leq T$}\}\subset \bigcup_{n_0}^{n_T} \mathcal Y(2^k).
\]


Let $\eta\asymp \rho$ be so that the map $g\mapsto xg$ is injective for all $x\in X_\rho$
and all  
\[
g\in{\rm Box}(\eta):=\exp(B_{i\mathfrak{sl}_2(\mathbb R)}(0,\eta))\exp(B_{\mathfrak{sl}_2(\mathbb R)}(0,\eta)).
\]

Fix some $1/2<s<1$ and some $z\in X$. We claim that  
\be\label{eq:num-conn-comp}
\#\bigl(\text{connected components of $\mathcal Y(2^k)\cap z.{\rm Box}(\eta)$}\bigr)\ll \alpha_s^{12/s}2^{6k/s}
\ee
where the implied constant depends on $\rho$.

For any connected component $C$ of $\mathcal Y(2^k)\cap z.{\rm Box}(\eta)$, there exists some
$v\in i\mathfrak{sl}_2(\mathbb R)$ so that 
\[
C=z\exp(v)\exp(B_{\mathfrak{sl}_2(\mathbb R)}(0,\eta)).
\]
Let us write $C=C_v$. 
Now in view of Theorem~\ref{cor:main-dist}, for every two connected components $C_v\neq C_{v'}$, we have 
\be\label{eq:separation-Cv}
\|v-v'\|\gg_{\rho} \alpha_s^{-4/s}2^{-2k/s}.
\ee
Because $\dim(\mathfrak r)=3$, the cardinality of an $\alpha_s^{-4/s}2^{-2k/s}$-separated
set in $B_{i\mathfrak{sl}_2(\mathbb R)}(0,\eta)$ is $\ll \alpha_s^{12/s}2^{6k/s}$, where the implied constant depends only on the choice of norm. 
The claim in~\eqref{eq:num-conn-comp} thus follows from~\eqref{eq:separation-Cv}.

Let $\bigl\{z_j.{\rm Box}(\eta): 1\leq j\leq R\bigr\}$ be a covering of $X_\rho$ with sets of the form $z.{\rm Box}(\eta)$;
we may find such a covering with $R=O(\Vol(X)\eta^{-6})$ the implied constant is absolute, see also the definition of $c_1$ in~\eqref{distance}. 
Then we compute
\begin{align*}
\mathcal N(2^k)&\leq2^{-k+1}\sum_{\mathcal Y(2^k)} \vol(xH)&&\text{by def.\ of $\mathcal Y(2^k)$}\\
&\ll 2^{-k}\sum_{j=1}^{M} \sum_{C_v\subset z_j.{\rm Box}(\eta)} \vol(C_v)&&\text{by~\eqref{eq:non-div-xH}}\\
&\ll \alpha_s^{12/s}\sum_{j=1}^{R} 2^{\frac{6k}{s}-k} && \text{by~\eqref{eq:num-conn-comp}} \\ & \ll {\Vol (X)}\alpha_s^{12/s}2^{\frac{6k}{s}-k}&&\text{since~$R=O(\Vol(X))$}
\end{align*}
in the above we also used the fact that $\vol(C_v)\ll_\rho 1$.

Since $ \rho\asymp \eta $ can be taken $\asymp \e_X^{k}$,  we conclude that for some absolute constant $N_1, N_2\ge 1$ and $c=c(s)\ge 1$,
\[
\mathcal N(T)\le c\, \Vol (X) \rho^{-N_1} \alpha_s^{12/s} \sum_{k=n_0}^{n_T} 2^{\frac{6k}{s}-k}\le 
c\, \Vol (X) \e_X^{-N_2} T^{\frac{6}{s}-1} 
\] 
which implies the claim (note here that $\Vol (X)=\Vol (M)$, since $\Gamma$ is torsion-free.) 
\end{proof}

\begin{Rmk}\label{volvol2}\rm  
Let $\mathcal N_M(T)$
be the number of properly immersed geodesic planes of area at most $T$ in a general geometrically finite manifold $M=\Gamma\ba \bH^3$. If $Y$ is a closed $H$-orbit $Y$
{\it{of finite area}} in $\Gamma\ba G$, then
 $\mathsf p_Y\asymp \mathsf s_Y=2$, $\tau_Y=\Vol (Y)$ and
 the non-divergence of the $U$-action as given in \cite[Thm. 1.1]{BZ}
implies that \eqref{eq:non-div-xH} also holds in this setting.
 
In view of these, the proof of Corollary \ref{volvol}
works in the same way for the following:
there exists $N\ge 1$ (depending only on $G$) such that for any $1/2<s<1$, we have
$$\mathcal N_M(T)\ll_s  \Vol (\text{unit-nbd of core} M)\, \e_M^{-N} T^{\frac{6}{s}-1}   $$
where the implied constant depends only on $s$. 
\end{Rmk}

\section{Appendix: Proof of Theorem \ref{main0} in the compact case}\label{sec:co-cpct-case}
In this section we present the proof of Theorem \ref{main0} when $X$ is compact.
As was mentioned in the introduction, this case is due to G.~Margulis.

Let $Y\neq Z$ be two closed $H$-orbits in $X=\Gamma\ba G$. Recall  $\e_X=\min_{x\in X}\op{inj}(x)$ where $\op{inj}(x)$ is the injectivity radius measured in $\Gamma\ba \bH^3$.

Fix $0<s<1$, and define $f_s:Y\to [2,\infty)$  as follows: for any $y\in Y $,
\[
f_s(y)=\begin{cases} \sum_{v\in I_Z(y)}\|v\|^{-s} & \text{ if $I_Z(y)\neq\emptyset $}\\
\e_X^{-s}&\text{otherwise}
\end{cases}
\]
where \[
I_Z(y)=\{v\in i\mathfrak{sl}_2(\mathbb R): 0<\|v\|<\e_X,\; y\exp(v)\in Z\}.
\]

Define $F_{s}=F_{s, Y, Z}:Y\to (0,\infty)$ as follows:
$$F_{s}(y)=f_{s}(y)+\Vol(Z) \e_X^{-s}.$$

 Note that in the case at hand, $F_s$ is a bounded Borel function on $Y$. 
  We also note that in the case at hand $\omega$, as defined in \eqref{defh}, is a bounded function on $X$ (recall that $\omega=2$
  in this case), and hence $F_s$ here and $F_{s,\lambda_s}$ that we considered in the proof of Theorem~\ref{main} are essentially the same functions in this case.

We use the following special case of Lemma~\ref{lem:vect-contr}: for any $v\in i\mathfrak{sl}_2(\mathbb R)$ with $\|v\|=1$, $1/3\le s<1$ and $t>0$, we have
\be\label{eq:linear-alg}
\int_{0}^1\frac{dr}{\|vu_ra_t\|^{s}} \le b_0 \frac{e^{(s-1)t/4}}{1-s}
\ee
where $vh= \Ad(h) (v)$ for all $h\in H$.
\begin{Rmk}\rm It is worth noting that  the symmetric interval $[-1,1]$ was used in Lemma~\ref{lem:vect-contr}. We remark that this is necessary in the infinite volume setting; indeed the half interval $[0,1]$ may even be a null set for $\mu_y$ for some $y$, 
see~\eqref{eq:def-ps-meas} for the notation.
\end{Rmk}

 For a locally bounded function $\psi$ on $Y$ and $t>0$, define
\be\label{attt}
{\mathsf A}_t\psi (y)= \int_{0}^1 \psi (yu_ra_t)dr \quad\text{for $y\in Y$}.
\ee

\begin{prop} Let $1/3\le s<1$.
There exists $t=t(s)>0$ such that for all $y\in Y$,
\be\label{eq:main-ineq-ccpt}
{\mathsf A}_t F_s(y) \le \frac{1}{2} F_s(y)+ c \, \e_X^{-4} \alpha_s^4  \op{Vol}(Z)\ee
where $\alpha_s=  (1-s)^{-1/(1-s)}$ and $c\ge 1$ is an absolute constant.
\end{prop}}
\begin{proof}
It suffices to show that 
${\mathsf A}_t f_s(y)\leq \tfrac{1}{2}f_s(y)+ \alpha_s^4 {\rm Vol}(Z)$.

Let $b_0$ be as in~\eqref{eq:linear-alg}, and
let $t=t(s)$ be given by the equation 
\[
b_0 \frac{e^{(s-1)t /4}}{1-s} =1/2.
\]

We compare $f_s(yu_ra_t)$ and $f_s(y)$ for $r\in [0,1]$.
Let $C_1\asymp e^{ t}$ be large enough so that $\|vh\|\leq C_1\|v\|$ for all $v\in i\frak{sl}_2(\br)$ and all 
\[
h\in \{a_\tau u_r: | r|<1, |\tau|\leq t\}.
\] 
Let $v\in I_Z(yu_ra_t)$ be so that $\|v\|< \e_X/C_1$. 
Then $\|va_{-t}u_{-r}\|\leq \e_X$; in particular, $va_{-t}u_{-r}\in I_Z(y)$.

In the following, if $I_Z(\cdot)=\emptyset$, the sum is interpreted as to equal to $\e_X^{-s}$.
In view of the above observation and the definition of $f_s$, we have
\begin{align}
\notag f_s(yu_ra_t)&=\sum_{v\in I_Z(yu_ra_t)}\|v\|^{-s}\\
\notag&=\sum_{v\in I_Z(yu_ra_t), \|v\|< \e_X/C_1}\|v\|^{-s}+\sum_{v\in I_Z(yu_ra_t), \|v\|\geq \e_X/C_1}\|v\|^{-s}\\
\label{eq:A-two-sums-F}&\leq \sum_{v\in I_Z(y)}\|vu_ra_t\|^{-s}+\sum_{v\in I_Z(yu_ra_t), \|v\|\geq \e_X/C_1}\|v\|^{-s}.
\end{align}
Moreover, note that $\# I_Z(y)\ll \e_X^{-3} {\rm Vol}(Z)$ (see the proof of Lemma \ref{lem:volume-bd}). Hence, 
\be\label{eq:boundary-terms-F-cpct}
\textstyle\sum_{\|v\|\geq \e_X/C_1}\|v\|^{-s}\ll C_1^s \e_X^{-4}  {\rm Vol}(Z) \ll \e_X^{-4}  e^{st} {\rm Vol}(Z).
\ee

We now average~\eqref{eq:A-two-sums-F} over $[0,1]$.
Then using~\eqref{eq:boundary-terms-F-cpct} and~\eqref{eq:linear-alg} we get
\[
{\mathsf A}_t f_s(y)\leq \tfrac{1}{2}f_s(y)+ O(e^{st} {\rm Vol}(Z)).
\]
As $ (1-s)^{-1/(1-s)} \asymp e^{st/4}$, this proves ~\eqref{eq:main-ineq-ccpt}. 
\end{proof}

Let $m_Y$ be the $H$-invariant probability measure on $Y$:

\begin{cor}\label{FG}  
We have
\[
m_Y(F_s) \le  c \, \e_X^{-4}\alpha_s^4 \op{Vol}(Z)
\]
where $c\ge 1$ is an absolute constant.
\end{cor}

\begin{proof}
Since $m_Y$ is an $H$-invariant probability measure, $m_Y({\mathsf A}_t f_s)=m_Y(f_s)$.
Hence the claim follows by integrating \eqref{eq:main-ineq-ccpt} with respect to $m_Y$. 
\end{proof}

\noindent{\bf Proof of Theorem \ref{main0}}.
There exists $\sigma>0$ such that for any $h\in B_H(\e_X)$ and $y\in Y$,
$F_s(y)\le \sigma F_s(yh)$ (cf. Lemma \ref{F});
$B_H(\e_X)$ denotes the $\e_X$-ball centered at the identity in $H$.

Hence, using  Corollary ~\ref{FG}, we deduce
\begin{align*}
f_s(y) &\le F_s(y)\le  \frac{\sigma\int_{B_H( \e_X)} F_s (yh) dm_Y(yh)}{m_Y(B(y, \e_X))}
\\ &\le  \frac{\sigma\cdot{m_Y(F_s)} }{m_Y(B(y, \e_X))}
\ll \alpha_s^4 \e_X^{-7} \op{Vol}(Y) \op{Vol}(Z)
\end{align*}
with an absolute implied constant.
Since
 $d(y, Z)^{-s}\le c_1 f_s(y)$ for an absolute constant $c_1\ge 1$ (see \eqref{fss}),
we have
\begin{equation}\label{fin0} d(y, Z)\gg \alpha_s^{-4/s} \e_X^{7/s}\op{Vol}(Z)^{-1/s} \op{Vol}(Y)^{-1/s}.
\end{equation}

This shows Theorem \ref{thm:lattice-intro}(1). 
By  Corollary \ref{FG} and the Chebyshev inequality, we get
\begin{multline*} 
m_Y\{y\in Y: d(y, Z)\le \e\} \le m_Y\{y\in Y: F_s(y) \ge c_1^{-1} \e^{-s}\}
 \le c_1 m_Y(F_s) \e^s.
 \end{multline*}
Therefore
 \begin{equation}\label{fin}
 m_Y\{y\in Y: d(y, Z)\le \e\} \le c_1 c \e^s \e_X^{-4}\alpha_s^4 \Vol (Z),
 \end{equation}
which implies Theorem \ref{thm:lattice-intro}(2). 
\qed

\end{document}